\documentclass[11pt]{amsart}
\usepackage{a4,amsmath,amssymb,amscd,amsfonts,verbatim}
\usepackage{color}

\newtheorem{prop}[equation]{Proposition}

\newtheorem{thm}[equation]{Theorem}
\newtheorem{cor}[equation]{Corollary}
\newtheorem{lem}[equation]{Lemma}

\theoremstyle{definition}

\newtheorem{defns}[equation]{Definitions}
\newtheorem{rem}[equation]{Remark}
\newtheorem{rems}[equation]{Remarks}

\newtheorem{exa}[equation]{Example}

\theoremstyle{remark}

\numberwithin{equation}{section}

\newcommand{\spandsp}{\mbox{$\qquad\text{and}\qquad$}}

\newcommand{\sands}{\mbox{$\quad\text{and}\quad$}}

\newcommand{\sts}[1]{\mbox{$\quad\text{#1}\quad$}}

\newcommand{\Ker}{\operatorname{Ker}}
\newcommand{\Ima}{\operatorname{Im}}
\newcommand{\Cok}{\operatorname{Cok}}

\newcommand{\Tor}{\operatorname{Tor}}

\newcommand{\letbe}{\mathbin{\raisebox{.32pt}{:}\!=}}

\newcommand{\br}[1]{\mbox{$\langle #1\rangle$}}

\newcommand{\?}{\mskip-2mu}

\newcommand{\+}{\mskip-.5mu}
\newcommand{\0}{\mskip-.8mu}
\newcommand{\1}{\mskip-1mu}
\newcommand{\2}{\mskip-1.5mu}
\newcommand{\om}{\mskip1mu}
\newcommand{\pfm}{\mskip.5mu}
\newcommand{\opfm}{\mskip1.5mu}
\newcommand{\scirc}{\mathbin{\1\raisebox{1.3pt}{$\scriptscriptstyle \circ$}}}
\newcommand{\idehat}[1]{\skew{6}\widehat{\rule{0ex}{1.4ex}\smash{#1}}}
\newcommand{\cdt}{\raisebox{2.2pt}{$\opfm.\opfm$}}
\newcommand{\zmt}{\mbox{$\bZ\1/\+2$}}

\newcommand{\KP}{\mbox{$\bK P$}}
\newcommand{\RP}{\mbox{$\bR P$}}

\newcommand{\RPT}{\mbox{$\bR P^2$}}
\newcommand{\RPn}{\mbox{$\bR P^n$}}
\newcommand{\KPn}{\mbox{$\bK P^n$}}
\newcommand{\RPI}{\mbox{$\bR P^\infty$}}
\newcommand{\RPIssp}{\mbox{$\bR P^\infty_\ssp$}}
\newcommand{\KPI}{\mbox{$\bK P^\infty$}}

\newcommand{\CPT}{\mbox{$\bC P^2$}}
\newcommand{\CPH}{\mbox{$\bC P^3$}}
\newcommand{\CPn}{\mbox{$\bC P^n$}}
\newcommand{\CPI}{\mbox{$\bC P^\infty$}}

\newcommand{\HPT}{\mbox{$\bH P^2$}}
\newcommand{\HPH}{\mbox{$\bH P^3$}}
\newcommand{\HPn}{\mbox{$\bH P^n$}}
\newcommand{\HPI}{\mbox{$\bH P^\infty$}}

\newcommand{\OPT}{\mbox{$\bO P^2$}}

\newcommand{\ssp}{{\scriptscriptstyle{+}}}

\newcommand{\GL}{\mbox{\it GL\/}}

\newcommand{\BSp}{\mbox{$B\1 Sp$}}
\newcommand{\Spin}{\mbox{\it Spin}}
\newcommand{\Pin}{\mbox{\it Pin}}
\newcommand{\Pind}{\mbox{\it Pin}$^\dagger$}

\newcommand{\Pindd}{\mbox{$\Pin^{\,\ddagger\0}(d)$}}
\newcommand{\Spindd}{\mbox{$\Spin^{\,\ddagger\0}(d)$}}

\newcommand{\MPd}{\mbox{$M\?P^d$}}

\newcommand{\BO}{\mbox{$B\1 O$}}

\newcommand{\bC}{\mathbb{C}}
\newcommand{\bH}{\mathbb{H}}
\newcommand{\bK}{\mathbb{K}}

\newcommand{\bO}{\mathbb{O}}

\newcommand{\bR}{\mathbb{R}}
\newcommand{\bZ}{\mathbb{Z}}

\newcommand{\opD}{\mbox{$D_{\+ o}$}}
\newcommand{\opd}[1]{\mbox{$D^{#1}_{\+ o}$}}

\newcommand{\srf}[1]{\mbox{${\text{\it SR}^F}$}}
\newcommand{\Th}{\mbox{\it Th}\opfm}

\newcommand{\spt}{\mbox{$S\0 P^2$}}
\newcommand{\sptssp}{\mbox{$S\0 P^2_{\?\ssp}$}}

\newcommand{\sptx}{\mbox{$\spt(X)$}}
\newcommand{\sptn}{\mbox{$S\0 P^2_n$}}

\newcommand{\sptsn}{\mbox{$\spt(S^n)$}}

\newcommand{\sptKPn}{\mbox{$SP^2(\bK P^n)$}}

\newcommand{\BFd}{\mbox{$B\0 F^d$}}

\newcommand{\BPd}{\mbox{$B\0 P^d$}}
\newcommand{\BFdssp}{\mbox{$B\0 F^d_{\ssp}$}}
\newcommand{\BPdssp}{\mbox{$B\0 P^d_{\ssp}$}}

\newcommand{\RKI}{\mbox{$RK^\infty$}}
\newcommand{\RKIssp}{\mbox{$RK^\infty_\ssp$}}
\newcommand{\RKn}{\mbox{$RK^n$}}
\newcommand{\fssp}{\mbox{$f^{\scriptscriptstyle +}$}}

\newcommand{\twa}[2]
{\begin{picture}(35.5,7)
\put(14.5,-7){$\scriptstyle#1$}
\put(3,-2){$\llla\joinrel\hspace{-1pt}\relbar$}
\put(3,2.5){$\relbar\joinrel\hspace{-1pt}\llra$}
\put(14.5,8){$\scriptstyle#2$}
\end{picture}}

\newcommand{\ra}{\rightarrow}
\newcommand{\lra}{\longrightarrow}
\newcommand{\llra}{\relbar\joinrel\hspace{-2pt}\lra}

\newcommand{\lla}{\longleftarrow}
\newcommand{\llla}{\lla\joinrel\hspace{-2pt}\relbar}

\makeatletter \newcommand{\Ddots}{\mathinner{\mkern1mu\raise\p@
\vbox{\kern7\p@\hbox{.}}\mkern2mu
\raise4\p@\hbox{.}\mkern2mu\raise7\p@\hbox{.}\mkern1mu}}
\makeatother

\begin{document}
\bibliographystyle{plain}

\title[Symmetric squares of projective space]
{On the symmetric squares of\\ complex and quaternionic projective space}

\address{School of Mathematics, University of Manchester, Oxford
Road, Manchester M13~9PL, England}
\author{Yumi Boote}
\email{yumi.boote@manchester.ac.uk}
{\thanks{The work of the first author was supported by EPSRC
grant EP/P505631/1}}
\author{Nigel Ray}
\email{nigel.ray@manchester.ac.uk}



\begin{abstract}
The problem of computing the integral cohomology ring of the symmetric
square of a topological space has been of interest since the 1930s,
but limited progress has been made on the general case until
recently. In this work we offer a solution for the complex and
quaternionic projective spaces $\mathbb{K}P^n$, by taking advantage of
their rich geometrical structure. Our description is in terms of
generators and relations, and our methods entail ideas that have
appeared in the literature of quantum chemistry, theoretical physics,
and combinatorics. We deal first with the case $\mathbb{K}P^\infty$,
and proceed by identifying the truncation required for passage to
finite $n$. The calculations rely upon a ladder of long exact
cohomology sequences, which arises by comparing cofibrations
associated to the diagonals of the symmetric square and the
corresponding Borel construction. These involve classic configuration 
spaces of unordered pairs of points in $\mathbb{K}P^n$, and their 
one-point compactifications; the latter are identified as Thom spaces by 
combining L\"owdin's symmetric orthogonalisation process (and its 
quaternionic analogue) with a dash of \emph{Pin} geometry. The 
ensuing cohomology rings may be conveniently expressed using 
generalised Fibonacci polynomials. We note that our conclusions are 
compatible with mod $2$ computations of Nakaoka and homological 
results of Milgram.

\end{abstract}

\maketitle

%
%
%
%
%
%
%
%
%

\section{Introduction}\label{in}

The study of cyclic and symmetric powers has a long and varied
history, and has remained active throughout the development of
algebraic topology since the 1930s. At first, symmetric squares of
smooth manifolds were associated mainly with critical point theory
\cite{mor:cvl}, but by the 1950s cyclic powers of simplicial complexes
had come to underlie Steenrod's construction of mod $p$ cohomology
operations \cite{step:co} and related work on the homology and
cohomology of the symmetric groups \cite{nak:dth}. These interactions
were extended to stable splittings of classifying spaces during the
1980s \cite{mipr:sps}. More recently \cite{kal:spd}, symmetric powers
of smooth manifolds $M$ have been viewed as compactifications of
configuration spaces of finite sets of distinct points on $M$, and
also as important examples of orbifolds \cite{adleru:ost}. Since 2000
or so, the latter perspective has gained popularity within theoretical
physics \cite{bohasc:sac}, where they have led to the theory of
\emph{cyclic} and \emph{permutation orbifolds}.
 
For any topological space $X$, its \emph{cyclic} or \emph{symmetric
  square} $\sptx$ is the orbit space $(X\times X)/\om C_2$ under the
action of the cyclic group $C_2$, generated by the involution $\iota$
that interchanges the factors. Its elements are the unordered
pairs $\{x,y\}$ for all $x,\,y$ in $X$, and the fixed points of
$\iota$ determine the \emph{diagonal} subspace
$\varDelta=\varDelta(X)$, which contains all pairs $\{x,x\}$ and is
homeomorphic to $X$.  If $X$ is a CW complex of finite type, then so
is $\sptx$, and this condition holds for every space considered
below. The quotient map $q\colon X\times X\to\spt(X)$ is a
\emph{ramified covering} in the sense of \cite{smi:trc}, and its
off-diagonal restriction
\[
q'\colon (X\times X)\setminus\varDelta\lra\spt(X)\setminus\varDelta
\]
is a genuine double covering. By definition, the codomain 
$\spt(X)\setminus\varDelta$ is the \emph{configuration space} 
$\mathcal{C}_2(X)$ of unordered pairs of distinct points of $X$. 

In a small number of special cases, $\sptx$ may be identified with
some other familiar space. For example, there are homeomorphisms
\[
\spt(S^1)\;\cong\;M,\quad\spt(\RPT)\;\cong\;\RP^4\sands
\spt(S^2)\;\cong\;\CPT,
\]
where $M$ is the closed M\"{o}bius band. In these cases the integral 
cohomology ring $H^*(\sptx)$ is well-known, but its evaluation 
presents a greater challenge for arbitrary $X$, and few complete 
calculations appear in the literature.

The most basic examples are the spheres $S^n$, which were studied by
Morse in \cite[Chapter VI \S11]{mor:cvl}, and by Nakaoka in a sequence
of papers including \cite{nak:ctc}. Spheres apart, the complex and
quaternionic projective spaces $\KPn$ are good candidates for families
of well-behaved CW complexes with rich internal geometry. Our main
algebraic result is the following distillation of Theorem
\ref{intcohsptn}, which describes the unreduced integral cohomology
ring $H^*(\spt(\KPn)_\ssp)$ by means of generators and relations. The
dimensions of the generators depend on the dimension $d$ of $\bK$ over
$\bR$; thus $d=2$ for $\bK=\bC$ or $4$ for $\bK=\bH$, and $g$, $h$,
and $u_{i,j}$ have dimensions $d$, $2d$, and $2i+jd+1$ respectively.
\begin{thm}\label{distil}
For any $n\geq 1$, the ring $H^*(\spt(\KPn)_\ssp)$ is isomorphic to
\[
\bZ\opfm[\om h^p\!/2^{p-1}\!,\,g^qh^s\!/2^s\!,\,u_{i,j}]\,/\,J_n\,,
\]
where $p,\om q\geq 1$, $s\geq 0$ and $0<i<jd\2/2$; the ideal $J_n$ 
is given by
\[
(2u_{i,j},\,u_{i,j}u_{k,l},\,u_{i,j}h^p\!/2^{p-1}\!,\,
u_{i,j}g^qh^s\!/2^s\!,
\, r_t,\,u_{i,t}:t>n)\,,
\]
for certain homogeneous polynomials $r_t=r_t(g,h)$ of dimension $td$.
\end{thm}
The torsion-free product structure is indicated by the form of the
generating monomials; for example, $g\cdot h=2\cdot gh\2/\1 2$
in $H^{3d}(\spt(\KPn))$, for any $n\geq 2$.

Since this manuscript was begun, two significant advances have been
made.  Firstly, Gugnin \cite{gug:icr} computed the torsion free
quotient ring $H^*(\sptx)/\Tor$ for any CW complex of finite type,
using the ring structure of $H^*(X)$ and the transfer homomorphism
\cite{smi:trc} associated to $q$. Secondly, Totaro \cite[Theorem
  1.2]{tot:ich} converted Milgram's description \cite{mil:hsp} of the
integral \emph{homology} groups $H_*(\sptx)$ into dualisable form, at
least for finite connected CW complexes whose integral cohomology
groups are torsion free and zero in odd dimensions (such an $X$ may be
called \emph{even}). It is straightforward to check that our
calculations are compatible with these developments, as well as with
Nakaoka's classic works \cite{nak:cpf} and \cite{nak:ctc}, which focus
mainly on $\zmt$ coefficients.

More, however, is true. The results of \cite{gug:icr} and
\cite{tot:ich} may be combined with those of \cite{nak:ctc} to show
that Theorem \ref{distil} is generic, insofar as it extends to
describing $H^*(\sptx)$ in terms of $H^*(X)$ for any even $X$. This
will be discussed in \cite{bora:sse}, and compared with the
geometrical approach in cases such as the octonionic projective
plane. The example $\spt(\OPT)$ is especially attractive, and offers
further evidence that $H^*(\sptx)$ may sometimes be understood by
direct appeal to the geometry of $X$. Weighted projective spaces are
also of interest in this context.

A crucial supporting r\^ole is played by the homotopy
theoretic analogue of $\sptx$, namely the \emph{Borel construction}
$S^\infty\2\times_{C_2}\?(X\times X)$, where $C_2$ acts antipodally,
and therefore freely, on the contractible sphere $S^\infty$. It
contains the diagonal subspace $\RPI\2\times X$, sometimes written 
as $\idehat{\varDelta}=\idehat{\varDelta}\om(X)$ below. There is a 
\emph{Borel bundle}
\begin{equation}\label{borbun}
X\times X\stackrel{k}{\lra}S^\infty\?\times_{C_2}\!(X\times X)
\stackrel{r}{\lra}\RPI,
\end{equation}
in which $\pi_1(\RPI)\cong C_2$ acts on the fibre by $\iota$. The
integral cohomology ring $H^*(S^\infty\2\times_{C_2}\?(X\times X))$
may often be computed via the Leray-Serre spectral sequence of
\eqref{borbun}, as for $X=\KPn$ in Theorem \ref{intcohbn} below.

There is also a canonical projection map
\begin{equation}\label{canproj}
\pi\colon S^\infty\?\times_{C_2}\!(X\times X) \lra\sptx,
\end{equation} 
which identifies $\idehat{\varDelta}$ with $\pi^{-1}(\varDelta)$; its
off-diagonal restriction
\begin{equation}\label{pinodiag}
\pi'\colon S^\infty\?\times_{C_2}\!(X\times X)
\setminus\idehat{\varDelta}\lra
\spt(X)\setminus\varDelta
\end{equation}
is a fibration with contractible fibres, and hence a homotopy
equivalence. The Vietoris-Begle theorem shows that $\pi$
induces an isomorphism
\[
\pi^*\colon H^*(\sptx;\bZ[1/2])\lra 
H^*(S^\infty\?\times_{C_2}\!(X\times X);\bZ[1/2]),
\]
so the main difficulties in proving Theorem \ref{intcohsptn} involve
$2$-torsion and $2$-primary product structure. Our assumptions on 
$X$ ensure that the induced map 
\begin{equation}\label{pipp}
\pi''\colon S^\infty\?\times_{C_2}\!(X\times X)
/\idehat{\varDelta}\lra\spt(X)/\varDelta
\end{equation}
of quotient spaces is also a homotopy equivalence, allowing
cohomological properties of the Borel construction to be related
directly to those of $\sptx$, as in Bredon \cite[Chapter
  VII]{bre:ict}.

For $X=\KPn$, this relationship brings $\Pin$ geometry into play.
Because $\KPn$ is closed, compact and admits a suitable metric,
$\varDelta$ is a canonical deformation retract of a certain singular
open neighbourhood \cite[Corollary 4.2]{kata:gfg}. It turns out that
$\spt(\KPn)\setminus\varDelta$ contains an analogous closed, compact
submanifold $\varGamma_n$, which is a deformation retract; it may be
imagined as an \emph{anti-diagonal}. The retraction is canonical
because $\spt(\KPn)\setminus\varDelta$ is diffeomorphic to the total
space of a $\Pindd$ vector bundle $\theta_n$ over $\varGamma_n$, where
$\ddagger$ stands for $c$ when $d=2$, or $-$ when $d=4$ \cite[\S
  2]{hakrte:nfm}. In terms of \eqref{pipp}, $\spt(\KPn)/\varDelta$ is
therefore the Thom space of $\theta_n$, whose cohomological structure
helps to unlock the $2$-primary information. The relationship between
$\Pin^-(4)$ and $\Pin^+(4)$ has been clarified in the physics
literature \cite[\S 5.3]{bedemogwkr:pgp}, but continues to cause
confusion elsewhere.

Our results are inspired by two sources. One is James, Thomas, Toda,
and Whitehead \cite{jathtowh:sss}, who describe $\spt(S^n)/\varDelta$
as a Thom space over $\RPn$; our approach applies equally well to this
situation, and recovers $H^*(\spt(S^n))$ for $n\geq 1$. The other is
Yasui \cite{yas:rsp}, who introduces Stiefel manifolds into his
determination of $H^*(\mathcal{C}_2(\CPn))$, and builds on Feder's mod
2 calculations \cite{fed:rsp}. Recently, Dominguez, Gonzalez and
Landweber \cite{dogola:icc} have computed $H^*(\mathcal{C}_2(\RPn))$,
so the quaternionic case of Theorem \ref{intcohbpdn} completes the
trio.

Since $H^*(\KPn)\cong\bZ[z]/(z^{n+1})$ is the truncation of
$H^*(\KPI)$, the possibility suggests itself of calculating 
$H^*(\spt(\KPI))$ and expressing  $H^*(\spt(\KPn))$ as an appropriate
quotient. This is indeed feasible, and occupies the second half of
Theorem \ref{intcohsptn}. The fact that  
$\varGamma:=\varGamma_\infty$ is a model for the classifying
space $B\Pindd$ provides a convenient point of departure for
\S \ref{clsppibu}.

There are seven subsequent sections as follows, and an Appendix of 
low dimensional tables.
\begin{enumerate}
\item[\ref{clsppibu}.]
Classifying spaces and $Pin(d)$ bundles
\item[\ref{orcose}.]
Orthogonalisation and cofibre sequences 
\item[\ref{chcl}.]
Characteristic classes
\item[\ref{motwco}.]
Mod 2 cohomology
\item[\ref{inco}.]
Integral cohomology
\item[\ref{motwtr}.]
Mod 2 truncation
\item[\ref{intr}.]
Integral truncation.
\end{enumerate}

The results for $\HPn$ appear in the first author's PhD thesis
\cite{boo:phd}, and were presented in August 2014 as a contributed
talk at the Seoul ICM. They were originally intended for application
to quaternionic cobordism; that work remains in progress, along
with extensions to other cohomology theories and higher cyclic
powers. 

The authors wish to thank Andrew Baker for his ongoing interest and 
encouragement, and Larry Taylor for helpful comments on 
\emph{Pin$^\pm\?(k)$}.

%
%
%
%
%
%
%
%
%

\section{Classifying spaces and \Pind$\2(d)$ bundles}\label{clsppibu}

This section establishes notation that allows the complex and
quaternionic cases to be treated simultaneously, and collates
background information on certain low dimensional compact Lie
groups. The aim is to contextualise $\Pin^c(2)$ and $\Pin^{\pm}(4)$;
the latter requires special care, having been the subject of several
ambiguities in the literature. Background sources include
\cite{por:cac} for quaternionic linear algebra, \cite{bak:tsc} for the
accidental isomorphisms from a quaternionic viewpoint, and
\cite{yas:rsp} for aspects of the complex case.

Henceforth, $\bK$ denotes the complex numbers $\bC$ or the quaternions
$\bH$, and $d$ stands for their respective dimensions $2$ or $4$, as
real vector spaces. Scalars act on the \emph{right} of vector spaces
over $\bK$, unless otherwise stated; thus $\GL(n,\bK)$ act on the
\emph{left} of $\bK^n$, whose elements are column vectors. By
definition, the compact Lie group $O_\bK(n)<\GL(n,\bK)$ consists of
all matrices $Q$ that preserve the Hermitian inner product
\[
\br{u,\om v}\;=\;\bar{u}_1v_1+\cdots +\bar{u}_nv_n\;=\;u^\star v
\]
on $\bK^n$, where ${}^\star$ denotes conjugate transpose. Such
matrices are characterised by the property that $Q^{-1}\?=Q^\star$; in
particular, $O_\bK(n)$ is the unitary group $U(n)$ or the symplectic
group $Sp(n)$, over $\bC$ or $\bH$ respectively.

When $n=2$ there exist important \emph{accidental isomorphisms}
\begin{equation}\label{accisos}
U(2)\;\cong\;\Spin^c(3)\spandsp Sp(2)\;\cong\;\Spin(5).
\end{equation}
They may be understood in terms of the real $(d+1)$-dimensional 
vector space
\[
\mathfrak{H\pfm}_2^0(\bK)\;=\;
\left\{\begin{pmatrix}r&k\\\bar{k}&-r
\end{pmatrix}
:(r,k)\in\bR\times\bK\right\}
\]
of $2\times 2$ Hermitian matrices with zero trace, on which $O_\bK(2)$
acts on the left by
\begin{equation}\label{spinact}
Q\cdot Z\;=\;Q\pfm Z\pfm Q^\star
\end{equation}
for any $Z$ in $\mathfrak{H\pfm}_2^0(\bK)$. This provides the action of
$\Spin^c(3)$ on $\bR^3$ or $\Spin(5)$ on $\bR^5$, although the former
is often given by the equivalent action on \emph{skew}-Hermitian
matrices.

For any $n\geq 2$, let $V_{n+1,2}$ denote the Stiefel manifold of
orthonormal $2$-frames in $\bK^{n+1}$; it is a closed compact manifold
of dimension \mbox{$(2n+1)d-2$}, whose elements are specified by
$(n+1)\times 2$ matrices $(u\:v)$ over $\bK$, with orthonormal
columns. The group $O_\bK(2)$ acts freely on the \emph{right}, and has
orbit space the Grassmannian $Gr_{n+1,2}$.  For $n=\infty$, the
colimit $V_2\letbe V_{\infty,2}$ is contractible, and $Gr_2\letbe
Gr_{\infty,2}$ serves as a classifying space $BO_\bK(2)$. In view of
the accidental isomorphism \eqref{accisos}, $BO_\bK(2)$ is
$B\Spin^c(3)$ or $B\Spin(5)$, and the associated real $(d+1)$-plane
bundle $\chi$ is induced by the action \eqref{spinact}. Of
course $BO_\bK(2)$ also admits the standard universal $\bK^2$ bundle
$\omega_2$, which is written as $\zeta_2$ over $BU(2)$ or $\xi_2$
over $\BSp(2)$ in \S\ref{chcl}.

Note that $O_\bK(1)$ is the unit sphere and multiplicative subgroup
$S^{d-1}<\bK^\times$, and $BO_\bK(1)$ is the projective space $\KPI$,
with tautological line bundle $\omega=\omega_1$; the latter is written
as $\zeta$ over $\CPI$ or $\xi$ over $\HPI$. Thus $S^{d-1}\2\times
S^{d-1}<O_\bK(2)$ is the subgroup of diagonal matrices, and isomorphic 
to $\Spin^c(2)$ or $\Spin(4)$.

\begin{defns}\label{ndzmtssubgps}
The subgroup $P^d<O_\bK(2)$ consists of all matrices 
\begin{equation*}
\left\{
\begin{pmatrix}
a&0\\
0&b
\end{pmatrix},
\begin{pmatrix}
0&a\\
b&0
\end{pmatrix}:
a,\om b\in S^{d-1}\right\};
\end{equation*}
the subgroups $F^d\letbe C_2\om\times S^{d-1}$ and 
$S^{d-1}\2\times S^{d-1}<P^d$ consist of matrices
\begin{equation*}
\left\{
\begin{pmatrix}
a&0\\
0&a
\end{pmatrix},
\begin{pmatrix}
0&a\\
a&0
\end{pmatrix}:
a\in S^{d-1}\right\}\sands
\left\{
\begin{pmatrix}
a&0\\
0&b
\end{pmatrix}:
a,\om b\in S^{d-1}\right\}
\end{equation*}
respectively, where $C_2$ is generated by
$\tau\letbe\left(\begin{smallmatrix}0&1\\1&0\end{smallmatrix}\right)$.
\end{defns}
\begin{rems}\label{remcosetsps}
The group $P^d$ is the wreath product $S^{d-1}\2\wr C_2$, and the
normalizer of $S^{d-1}\2\times S^{d-1}$ in $O_\bK(2)$; so it is
isomorphic to the compact Lie group \smash{$\Pindd$}, where $\ddagger$
stands for $c$ when $d=2$, or $-$ when $d=4$ \cite[\S
  2]{hakrte:nfm}. This motivates the notation, and holds because
$\Pindd$ is the normalizer of $\Spindd$ in
\mbox{$\Spin^{\,\ddagger\0}(d+1)$}, and $\Spin^{\,-\0}=\Spin$.  The
quotient epimorphism $P^d\to C_2$ is the composition of the double
covering $\Pindd\to O(d)$ with the determinant.
 \end{rems}
  
The left action \eqref{spinact} of $P^d$ on $\bR\times\bK$ is made 
explicit by 
\begin{equation}\label{spinew}
\tau\cdot(r,k)\;=\;(-r,\bar{k})\spandsp
\begin{pmatrix}a&0\\0&b\end{pmatrix}\cdot 
(r,k)\;=\;(r,ak\bar{b})\,.
\end{equation}
This splits as the product of the actions on $\bR$ by $\det$, and on
$\bR^d\cong\bK$ by
\begin{equation}\label{spinex}
\tau\cdot k\;=\;\bar{k}\spandsp
\begin{pmatrix}a&0\\0&b\end{pmatrix}\cdot 
k\;=\;ak\bar{b}\,.
\end{equation}
\begin{prop}\label{cosetsps}
There are homeomorphisms of left coset spaces
\begin{gather*}
{\bf (1)}\; O_\bK(2)/(S^{d-1}\2\times S^{d-1})\;\cong\; S^d,
\quad{\bf (2)}\; O_\bK(2)/P^d\;\cong\;\bR P^d,\\
{\bf (3)}\; P^d/F^d\;\cong\;S^{d-1}
\;\;\,\text{and}\quad
{\bf (4)}\; P^d/(S^{d-1}\2\times S^{d-1})\;\cong\;C_2. 
\end{gather*}
\end{prop}
\begin{proof}
For (1), observe that the left action \eqref{spinact} of $O_\bK(2)$ on
$S^d\subset\bR\times\bK$ is transitive, and $(1,0)$ has isotropy
subgroup $S^{d-1}\2\times S^{d-1}$. The induced action of $O_\bK(2)$
on $\bR P^d$ is also transitive, and (2) holds because
$\tau\cdot(1,0)=(-1,0)$, so $[1,0]$ has isotropy subgroup $P^d$.  For
(3), note that the left action \eqref{spinex} of $P^d$ on
$S^{d-1}\subset\bK$ is transitive, and $1$ has isotropy subgroup
$F^d$. Finally, (4) is the isomorphism of topological groups induced by 
the quotient epimorphism.
\end{proof}

Any closed subgroup $H$ of a compact Lie group $G$ gives rise to 
a bundle 
\[
G/H\stackrel{i}{\lra}BH\stackrel{p}{\lra}BG,
\]
in which $p$ is modelled by the natural projection $EG/H\to EG/G$;
see \cite[Chapter II]{mito:tlg}, for example. So Proposition 
\ref{cosetsps} has the following consequences.
\begin{cor}\label{fibrns}
There exist bundles
\begin{gather*}
{\bf (1)}\; S^d\lra B(S^{d-1}\2\times
S^{d-1})\stackrel{p_1}{\lra}BO_\bK(2), \;\; {\bf (2)}\; \bR
P^d\stackrel{i_2}{\lra} \BPd\stackrel{p_2}{\lra}BO_\bK(2),\\ {\bf
  (3)}\; S^{d-1}\lra \BFd\stackrel{p_3}{\lra}\BPd, \;\;{\bf (4)}\;
C_2\lra B(S^{d-1}\2\times S^{d-1})\stackrel{p_4}{\lra}\BPd,
\end{gather*}
where {\rm (1)} is the sphere bundle of $\chi$, {\rm (2)} is the
projectivisation of $\chi$, {\rm (3)}~is the sphere bundle of
the universal $\Pindd$ vector $d$-plane bundle $\theta$, and {\rm (4)} 
is the double covering associated to the determinant line bundle
$\det(\theta)$.
\end{cor}
\begin{proof}
The homeomorphisms of Proposition \ref{cosetsps} associate the 
left actions of $O_\bK(2)$ on the coset spaces (1) and (2) to their 
left actions on $S^d$ and $\om\bR P^d$, as given by 
\eqref{spinact}; and the left actions of $P^d$ on the coset spaces 
(3) and (4) to their left actions on $S^{d-1}$ and $\om C_2$, as 
given by Remarks \ref{remcosetsps} and \eqref{spinex}.
\end{proof}

\begin{rems}
The proof of Corollary \ref{fibrns} uses the orbit space  
\[
\varGamma\;\letbe\;V_2/P^d\;=\;
\left\{[u],[v]\,:\,u,v\in S^\infty,\;u^\star v=0\right\}
\]
as the model for $\BPd$. Its elements are unordered pairs of
orthogonal lines in $\bK^\infty$, so $\varGamma$ may be interpreted 
as a subspace of $\spt(\KPI)$. The corresponding models for 
$B(S^{d-1}\?\times\2 S^{d-1})$ and 
$\BFd$ are $V_2/(S^{d-1}\?\times\2 S^{d-1})$ and $V_2/F^d$ 
respectively.
\end{rems}
\begin{rems}\label{splitlam}
The universal $\Pindd$ structure on $\theta$ corresponds to the 
induced $\Spin^{\,\ddagger\0}(d+1)$ structure on 
$\det(\theta)\oplus\theta$, via the isomorphism
\begin{equation}\label{pindetsplit}
p_2^*(\chi)\;\cong\;\det(\theta)\oplus\theta
\end{equation}
associated to the splitting \eqref{spinex} of \eqref{spinew}. In the
quaternionic case, this exemplifies \cite[Lemma 1.7]{kita:psl};
moreover, $\Pin^-(4)$ and $\Pin^+(4)$ are isomorphic \cite[\S
  5.3]{bedemogwkr:pgp}, so $\theta$ must not be confused with the
universal $\Pin^+(4)$ bundle over $\varGamma$. The latter involves 
modifying \eqref{spinex} by $\tau\cdot k=-\bar{k}$, for any
$k$ in $\bH$ \cite{boo:phd}.

In either case, the bundle of Corollary \ref{fibrns}(2) arises from
the bundle (1) by factoring out the action of $\tau$, so the double
covering (4) is also the $S^0$-bundle of the tautological real line
bundle $\lambda$ over $\RP(\chi)$. Thus \eqref{pindetsplit} coincides
with the standard splitting of $p_2^*(\chi)$ as
$\lambda\oplus\lambda^\perp$.
\end{rems}

The universal $\Pindd$ disc bundle may be displayed as the
diagram
\begin{equation}\label{daddiag}
S(\theta)\stackrel{\subset}{\llra}
D(\theta)\twa{\raisebox{1pt}{$\scriptstyle\supset$}}{s}\varGamma,
\end{equation}
where $s$ is the projection and the zero-section its left inverse. By
\eqref{spinex} this is homeomorphic to the diagram
\begin{equation}\label{saddad}
V_2\times_{P^d}\2 S^{d-1}\stackrel{\subset}{\llra}
V_2\times_{P^d}\2 
D^d\twa{\raisebox{1pt}{$\scriptstyle\supset$}}{s}V_2/P^d,
\end{equation}
where $P^d$ acts on the unit disc $D^d\subset\bK$ by $\tau\cdot
q=\bar{q}$ and
$\left(\begin{smallmatrix}a&0\\0&b\end{smallmatrix}\right)\cdot
  q=aq\bar{b}$. Then Proposition \ref{cosetsps}(3) identifies
  $S(\theta)$ as a model for $\BFd$.

The \emph{open} disc bundle  associated to $\theta$ has total space 
$D_o(\theta)=V_2\times_{P^d}\2\opd{d}$, and fibre the open unit 
disc $\opd{d}\subset\bK$.

\begin{defns}\label{defgamman}
For any $n\geq 1$, the smooth manifold $\varGamma_n\subset\varGamma$
is the orbit space $V_{n+1,2}/P^d$, of dimension $(2n-1)d$. Its
elements are unordered pairs of orthogonal lines in $\bK^{n+1}$, so
$\varGamma_n$ may be interpreted as a subspace of $\sptKPn$; it also
admits the $\Pindd$ bundle $\theta_n$, which is the restriction of
$\theta$.
\end{defns}

By construction, the inclusion $\varGamma\subset\spt(\KPI)$ is 
the colimit of the inclusions $\varGamma_n\subset\sptKPn$, and 
\eqref{daddiag} is the colimit of the disc bundle diagrams
\[
S(\theta_n)\stackrel{\subset}{\llra}D(\theta_n)
\twa{\raisebox{1pt}{$\scriptstyle\supset$}}{s_n}\varGamma_n\,.
\]
Similarly, $D_o(\theta)$ is the colimit of the total spaces 
$D_o(\theta_n)$.

%
%
%
%
%
%
%
%
%

\section{Orthogonalisation and cofibre sequences}\label{orcose}

Inspiration for this section is provided by \cite{jathtowh:sss}, 
where the cofibre sequence 
\begin{equation}\label{jathtowhcs}
S^n\stackrel{i_\varDelta}{\llra}\sptsn\stackrel{b_\varDelta}{\llra}
\Th(\eta_n^\perp)\llra\dots
\end{equation}
is introduced; $i_\varDelta$ denotes the inclusion of the diagonal
$S^n$, and $\Th(\eta_n^\perp)$ is the Thom space of the complement
of the tautological line bundle $\eta_n$ over $\RPn$, for $n\geq
1$. An analogue of \eqref{jathtowhcs} is developed for
$\sptKPn$, and $\Th(\theta_n)$ is identified with the $1$-point
compactification of the configuration space $\mathcal{C}_2(\KPn)$.

Henceforth, it is convenient to denote $\sptKPn$ by $\sptn$. As above,
$\KPn$ is taken to be the quotient space
\smash{$S^{(n+1)d-1}\?/S^{d-1}$} of the unit sphere; so for any $[y]$,
$[z]$ in $\KPn$, the real number $|y^\star z|$ is well-defined, and
determines a continuous function $\epsilon\colon \sptn\to[0,1]$.  By
definition, $\epsilon$ takes value $1$ on the diagonal $\KPn$, and $0$
on the subspace $\varGamma_n$ of Definition \ref{defgamman}. The
quantity $(1-|y^\star z|^2)^{1/2}$ is known as the \emph{chordal
  distance} from $[y]$ to $[z]$ in $\KPn$.

Following \S\ref{in}, the diagonal
$\varDelta_n=\KPn\subset\sptn$ has an open neighbourhood associated to
the tangent bundle $\tau_{\bK P^n}$. Its fibre is the open cone on
$\bR P^{nd-1}$, and it may be identified with
$\epsilon^{-1}(0,1]$. Now our aim is to find a similar link
between $\epsilon^{-1}[0,1)$ and $\theta_n$, for which a brief 
diversion is required.

\begin{defns} 
For any $h$ in $\opd{d}\subset\bK$, the positive real number $R$ 
and the positive definite Hermitian matrix $M_h$ are given by
\[
R\letbe(1+|h|)^{1/2}+(1-|h|)^{1/2}\spandsp
M_h\;\letbe\;
\begin{pmatrix}
R^2\!/2&h\\[.1em]
\bar{h}&R^2\!/2
\end{pmatrix}\?\2R^{-1}.
\]
So $R^2+4R^{-2}|h|^2=4$ and $(1-|h|^2)^{1/2}=1-2R^{-2}|h|^2$,
whence
\begin{equation}\label{mhmhinv}
M_h^2=
\begin{pmatrix}1&h\\\bar{h}&1\end{pmatrix}
\sands M_h^{-1}=
\begin{pmatrix}
R^2\!/2&-h\\[.1em]
-\bar{h}&R^2\!/2
\end{pmatrix}\?\2R^{-1}\1\big(1-|h|^2\big)^{-1/2}.
\end{equation}
The matrix $(u_h\:v_h)\letbe(u\:v)\opfm M_h$ lies in the
\emph{non-compact} Stiefel manifold 
\[
W_{n+1,2}\;\letbe\;
\left\{(w\: x)\,:\,\|w\|=\|x\|=1,\;|w^\star x|<1\right\}
\]
of normalised $2$-frames over $\bK$, because $u_h^\star v_h=h$ 
for any $(u\:v)$ in $V_{n+1,2}$. 
\end{defns}

The right action of $P^d$ on the subspace $V_{n+1,2}$ extends to 
$W_{n+1,2}$ by
\begin{equation}\label{ndactsonvt}
(w\:x)\cdot\tau\;=\;(x\:w)\spandsp(w\:x)\cdot
\begin{pmatrix}a&0\\0&b\end{pmatrix}\;=\;(wa\;xb),
\end{equation}
and the orbit space $W_{n+1,2}/P^d$ may be identified with the
subspace $\sptn\setminus\varDelta_n$ of $\sptn$. As such, it is an
open neighbourhood of $\varGamma_n$.

\begin{thm}
For any $n\geq 1$ (including $\infty$), there is a homeomorphism
\[
m\colon(\opD(\theta_n)\om,\om\varGamma_n)\;\lra\;
(\sptn\setminus\Delta_n\om,\om\varGamma_n)
\]
of pairs, which induces a deformation retraction of 
$\sptn\setminus\Delta_n$ onto $\varGamma_n$.
\end{thm}
\begin{proof}
The formula $f'((w\:x),\om h)\letbe(w_h\;x_h)$ defines a function
\[
f'\colon V_{n+1,2}\2\times\2\opd{d}\lra W_{n+1,2},
\]
which is continuous because $M_h$ varies continuously with $h$. The 
equations 
\[
M_h\om\tau\;=\;\tau\om M_{\bar{h}}\spandsp
M_h\begin{pmatrix}a&0\\0&b\end{pmatrix}\;=\;
\begin{pmatrix}a&0\\0&b\end{pmatrix}M_{\bar{a}hb}
\]
hold for all $a$, $b$ in $S^{d-1}$, so $f'$ is equivariant with respect to 
the free $P^d$-actions of \eqref{saddad} on $V_{n+1,2}\2\times\2\opd{d}$ 
and \eqref{ndactsonvt} on $W_{n+1,2}$. Moreover,{}   
\smash{$f''(w\:x)\letbe((w\:x)\om M_k^{-1};k)$} defines an 
equivariant inverse, where $k\letbe w^\star\pfm x$ satisfies $|k|<1$. So 
$f'$ induces a homeomorphism 
\smash{$f\colon\opD(\theta_n)\to\sptn\setminus\Delta_n$}
of $P^d$-orbit spaces. It also acts as the identity on 
the subspace of elements $((w\;x),0)$, which descends to the
zero section $\varGamma_n$ of $\opD(\theta_n)$. 

Since \smash{$|u_h^\star v_h|=t$} if and only if
$|h|=t$, the required retraction is defined by $l_t\letbe f\scirc
t\scirc f^{-1}$, where $t$ denotes fibrewise multiplication by
$t$ in $\opD(\theta_n)$ for $0\leq t\leq 1$; in particular, 
$l_0=s_n$. 

For both statements, the case $n=\infty$ is obtained by taking colimits.
\end{proof}
\begin{rem}
The homeomorphism $f''$ exemplifies L\"{o}wdin's \emph{symmetric
  orthogonalisation} procedure \cite{low:npc}, which originally arose
in the literature of quantum chemistry, albeit over $\bC$. There seem
to be no explicit references over $\bH$, but the proof remains valid
because quaternionic matrices have polar forms \cite{zha:qmq}.
For any normalised $2$-frame $(w\:x)$, where $w^\star x=k$ and 
$|k|<1$, the $2$-frame
\[
(w^k\:x^k)\;\letbe\;(w\:x)\om M_k^{-1}
\]
is orthonormal, and its construction is invariant with respect to
interchanging vectors in each frame. The procedure works 
because $M_k$ is the unique positive definite square root of 
$(w\:x)^\star(w\:x)$, as confirmed by \eqref{mhmhinv}.
\end{rem}
\begin{cor}\label{dadhalg}
For any $n\geq 1$, the configuration space $\mathcal{C}_2(\KPn)$ is 
homeomorphic to $\opD(\theta_n)$, and homotopy equivalent to 
$\varGamma_n$.\hfill$\Box$
\end{cor}
Since the quotient space $\sptn/\varDelta_n$ is homeomorphic to the
one-point compactification of $\mathcal{C}_2(\KPn)$, Corollary 
\ref{dadhalg} shows that both may be identified with the Thom space
$\Th(\theta_n)$, for any $n\geq 1$.

Now recall the canonical projection $\pi\colon B_n\to\sptn$ of
\eqref{canproj}, where $B_n$ denotes the Borel construction
$S^\infty\2\times_{C_2}(\KPn\?\times\KPn)$. Following
\eqref{pinodiag}, its restriction to the diagonal
\smash{$\idehat{\varDelta}_n$} is the projection
$\pi_2\colon\RP^\infty\?\times\KP^n\to\KPn$ onto the second
factor. For convenience, $\RP^\infty\?\times\KP^n$ will usually be
abbreviated to $\RKn$.
\begin{prop}
For every $n\geq 1$, the map $\pi$ induces a commutative ladder
\begin{equation}\label{thomlad}
\begin{CD}
\RKn@>{i_n}>>B_n@>{b_n}>>\Th(\theta_n)@>>>\!\dots\\[-.1em]
@V{\pi_2}VV@VV{\pi}V@VV{id}V\\[-.1em]
\KPn@>{i_\varDelta}>>\sptn@>{b_\varDelta}>>\Th(\theta_n)@>{}>>\!\dots
\end{CD}\quad
\end{equation}
of homotopy cofibre sequences.
\end{prop}
\begin{proof}
Both $i_n$ and $i_\varDelta$ are CW inclusions, and  $\pi$ is an 
off-diagonal homotopy equivalence, as noted in \eqref{pinodiag}. It 
therefore induces a map 
\begin{equation}\label{cofibdiag}
\begin{CD}
\RKn@>{i_n}>>B_n\!@>{q_n}>>\!\!B_n/\idehat{\varDelta}_n
\!\!@>>>\!\dots\\[-.1em]
@V{\pi\smash{_2}}VV\;@VV{\pi}V\!\!@VV{\pi\smash{''}}V\\[-.1em]
\KPn@>{i_\varDelta}>>\sptn@>{q_\varDelta}>>\!\sptn/\varDelta_n
\!\!@>>>\!\dots
\end{CD}
\end{equation}
of cofibre sequences, where $\pi''$ is a homotopy equivalence. To
complete the proof, define $b_n$ to be the composition $\pi''\scirc q_n$.
\end{proof}
\begin{rem}
The lower row of \eqref{thomlad} is our promised analogue of 
\eqref{jathtowhcs}.
\end{rem}
The corresponding ladder of colimits is given by
\begin{equation}\label{cofibsi}
\quad\;
\begin{CD}
\RKI@>{i}>>B@>{b}>>\MPd@>>>\!\dots\\[-.1em]
@V{\pi_2}VV@VV{\pi}V\,@VV{id_{\phantom{\varGamma}}}V\\[-.1em]
\KPI@>{i_\varDelta}>>\?\!\spt\!\?@>{b_\varDelta}>>\MPd@>{}>>\!\dots
\end{CD}\;\;,
\end{equation}
where $\MPd$ (more commonly written as $M\?Pin^\ddagger(d)$) denotes
the Thom space of the $\Pindd$ bundle $\theta$. Ladder \eqref{cofibsi}
also commutes, and its rows are cofibre sequences; it is the primary
input for the cohomology calculations of the remaining sections. If so
preferred, the crucial properties of \eqref{cofibdiag} may be deduced
directly from Bredon's results \cite[Chapter VII]{bre:ict}, using
\v{C}ech cohomology.

The upper row of \eqref{cofibsi} may be replaced by the homotopy cofibre 
sequence
\begin{equation}\label{bfbpmp}
\BFd\stackrel{p_3}{\llra}\BPd
\stackrel{j}{\llra}\MPd\llra\dots\,,
\end{equation}
where $\MPd$ is homeomorphic to the mapping cone of $p_3$, and $j$ 
denotes the inclusion of the zero section. The resulting ladder
is homotopy commutative, because there are homotopy equivalences $h_1$
and $h_2$ for which the square
\begin{equation}\label{h1h2sq}
\begin{CD}
\BFd@>{p_3}>>\BPd\\[-.1em]
\!@V{h_1}VV\;@VV{h_2}V\\[-.1em]
\RKI\!@>{i}>>B
\end{CD}
\end{equation}
is homotopy commutative. The existence of $h_1$ and $h_2$ involves 
standard manipulations with models for classifying spaces, following 
\cite{mit:npb}, for example. Similar arguments with $h_2$ lead to a 
homotopy equivalence between fibrations
\begin{equation}\label{clp4borbun}
B(S^3\1\times\2 S^3)\stackrel{p_4}{\lra}\BPd\stackrel{}{\lra}\RPI
\!\sands
\KPI\?\times\KPI\stackrel{k}{\lra}B\stackrel{}{\lra}\RPI\,;
\end{equation}
the former arises by classifying \eqref{fibrns}(4) and the latter
is the Borel bundle \eqref{borbun}.

%
%
%
%
%
%
%
%
%

\section{Characteristic classes}\label{chcl}

In this section, characteristic classes are determined for various of
the vector bundles introduced above. The results are expressed in the
notation of \S \ref{clsppibu}, and play an important part in the final
calculations.

The integral cohomology rings  
\begin{equation}\label{defzlolt}
H^*(BO_\bK(1)_\ssp)\;\cong\;\bZ[z]\spandsp 
H^*(BO_\bK(2)_\ssp)\;\cong\;\bZ[l_1,l_2]
\end{equation}
are standard, as are the properties of their generators. In
particular, $z$ has dimension $d$, and is the 1st Chern class
$c_1(\zeta)$ or the 1st symplectic Pontryagin class $p_1(\xi)$;
similarly, $l_1$ and $l_2$ have dimensions $d$ and $2d$, and are the
1st and 2nd Chern classes of $\zeta_2$ or the 1st and 2nd symplectic
Pontryagin classes of $\xi_2$ respectively. Tensoring with $\zmt$
yields the corresponding rings with mod $2$ coefficients, so $z$,
$l_1$ and $l_2$ may be confused with their mod $2$ reductions by 
allowing the context to distinguish between them. With this convention, 
the non-zero Stiefel-Whitney classes of the stated bundles are given by
\[
w_d(\omega)\;=\;z,\quad w_d(\omega_2)\;=\;l_1,\quad 
w_{2d}(\omega_2)\;=\;l_2\sands w_d(\chi)=l_1
\] 
in $H^*(BO_\bK(1);\zmt)$ and $H^*(BO_\bK(2);\zmt)$ respectively.

Our calculations also involve $\RPI$, and its tautological
real line bundle $\eta$. The integral and mod $2$ cohomology rings are
given by
\begin{equation}\label{defca}
H^*(\RPIssp)\;\cong\;\bZ[c]/(2c)\spandsp 
H^*(\RPIssp;\zmt)\;\cong\;\zmt\om[a],
\end{equation}
where $c=c_1(\eta_\bC)$ has dimension $2$ and $a=w_1(\eta)$ has
dimension $1$. In this case, reduction mod $2$ satisfies
$\rho(c)=a^2$. Because $H^*(\KPI)$ is torsion free, it follows from
the K\"unneth formula that the integral and mod $2$ cohomology rings
of $\RKI$ are given by
\begin{equation}\label{cohrki}
H^*(\RKIssp)\;\cong\;\bZ[c,z]/(2c)\spandsp 
H^*(\RKIssp;\zmt)\;\cong\;\zmt\om[a,z],
\end{equation}
where $a$, $c$, and $z$ are pullbacks of their namesakes along the 
projections.
 
The cohomology rings of $\KPn$, $Gr_{n+1,2}$, $\RPn$, and 
$\RPI\?\times\KPn$ are obtained from \eqref{defzlolt}, \eqref{defca}, 
and \eqref{cohrki} by appropriate truncation.

Remarks \ref{remcosetsps} identify the epimorphism $P^d\ra C_2$ with
$\det$, and Remarks \ref{splitlam} show that the line bundle $\lambda$
is classified by the induced map
\[
B\?\det\colon\BPd\lra\RPI,
\] 
which will also be denoted by $\lambda$. In particular, $w_1(\lambda)$
is given by $\lambda^*(a)$ in $H^1(\BPd;\zmt)$, and $c_1(\lambda_\bC)$
by $\lambda^*(c)$ in $H^2(\BPd)$. By definition, $\lambda$ restricts
to $\eta_d$ over the fibre $\bR P^d$ of Corollary \ref{fibrns}(2), so
$i_2^*(w_1(\lambda))=a$ in $H^1(\bR P^d;\zmt)$ and
$i_2^*(c_1(\lambda_\bC))=c$ in $H^2(\bR P^d)$. Henceforth,
$w_1(\lambda)$ and $c_1(\lambda_\bC)$ are written as $a$ and $c$
respectively.
\begin{rem}\label{acnonnilp}
The epimorphism $\det$ restricts to the identity on the subgroup
$C_2<P^d$, so $H^*(\RPI;R)$ is a summand of $H^*(\BPd;R)$ for 
coefficients $R=\bZ$ or $\zmt$. All powers of the elements $a$ and $c$ 
are therefore non-zero.
\end{rem}  

Now consider the pullback of $\omega_2$ along the projection
$p_2\colon \BPd\to BO_\bK(2)$ of Corollary \ref{fibrns}(2), and its
characteristic classes $x\letbe p_2^*(l_1)$ and $y\letbe p_2^*(l_2)$.
Equating the total Stiefel-Whitney classes of \eqref{pindetsplit}
gives
\[
1+x\;=\;(1+a)\,w(\theta).
\]
in $H^*(\BPd;\zmt)$. Since $w_i(\theta)=0$ for $i>d$, it follows that
\begin{equation}\label{stwhthetad}
w(\theta)\;=\;1+a+\dots+a^d+x,
\end{equation}
and hence that $a^{d+j}=a^jx$ in $H^{d+j}(\BPd;\zmt)$ for every 
$j\geq 1$. 
\begin{rems} 
These formulae highlight the importance of the integral cohomology
class $m\letbe c^{d/2}+x$, whose mod $2$ reduction $a^d+x$ will also
be written as $m$.  Then equation \eqref{stwhthetad} confirms that
$w_2(\theta)=m$ is the reduction of an integral class when $d=2$, and
that $w_2(\theta)=w_1^2(\theta)$ (because both are equal to $a^2$)
when $d=4$. These are defining properties for $\Pindd$-bundles
\cite[\S 2]{hakrte:nfm}.
\end{rems}

The map $p_2$ imposes an $H^*(BO_\bK(2)_\ssp;R)$-algebra structure on
$H^*(\BPdssp;R)$. For $R=\zmt$, the Leray-Hirsch theorem proves that
$p_2^*$ injects $\zmt[l_1,l_2]$ as a direct summand \cite{hus:fb}, and
that $H^*(\BPdssp;\zmt)$ has basis $1$, $a$, \dots, $a^d$ thereover.
Combined with \eqref{stwhthetad}, this identifies $H^*(\BPdssp;\zmt)$
as the $\zmt$ algebra
\begin{equation}\label{isobst}
G^*\om\letbe\;\zmt\om[\om a,m,y]/(am).
\end{equation}

The integral characteristic classes of $\Pindd$-bundles are equally 
important.
\begin{thm}\label{cohbpd}
The integral cohomology ring $H^*(\BPdssp)$ is isomorphic to
\[
Z^*\;\om\text{\rm $\letbe$}\;\;\bZ\om[\om c,m,y]/(2c,cm).
\]
\end{thm}
\begin{proof}
Consider the Leray-Serre spectral sequence
\[
E_2^{p,q}\;\letbe\; H^p(BO_\bK(2)_\ssp;H^q(\bR P^d_\ssp))
\;\Longrightarrow\; H^{p+q}(\BPdssp)
\]
of the fibration $p_2$, noting that $BO_\bK(2)$ is simply-connected. 
Since $H^*(BO_\bK(2))$ is free abelian and even dimensional, there are 
isomorphisms
\[
E_\infty^{*,*}\;\cong\;E_2^{*,*}\;\cong\; 
H^*(BO_\bK(2)_\ssp)\otimes H^*(\bR P^d_\ssp);
\]
thus $p_2^*$ is monic, and $i_2^*(c)=c$ in $H^2(\RP^d_\ssp)$. So any
additive generator $x^iy^j\otimes c^k$ of $E_\infty^{*,*}$ is
represented by $x^iy^jc^k$ in $H^*(\BPdssp)$, and there is
an isomorphism $E_\infty^{*,*}\cong H^*(\BPdssp)$ of
$H^*(BO_\bK(2)_\ssp)_+$-modules, on generators $1$, $c$, $\dots$,
\smash{$c^{d/2}$}.

The multiplicative structure is described by the single relation
$c^{d/2+1}=cx$, or $cm=0$, which follows from Remark \ref{acnonnilp}
and the fact that $cx\neq 0$.
\end{proof}
\begin{rems}\label{importacxy}
The homotopy commutative square \eqref{h1h2sq} may be used to 
transport the cohomology classes $c,\,z$, and $a$ to $H^*(\BFd)$, 
and $c,\,a,\,x$ and $y$ to $H^*(B)$ (surpressing $h_1^*$ and 
$h_2^*$ from the notation). The homotopy equivalence of 
\eqref{clp4borbun} ensures that $x$ and $y$ are then characterised by 
the fact that they satisfy
\[
k^*(x)=z_1+z_2\sands k^*(y)=z_1z_2
\]
in $H^*(\KPI\?\times\KPI)$, and are pullbacks from $H^*(\BO_{\bK}(2))$. 
These conventions lead to isomorphisms $H^*(\BFdssp)\cong\bZ[c,z]/(2c)$ 
and $H^*(B_{\ssp})\cong Z^*$. 

Similarly, bundles $\eta$ and $\omega$ are defined over $\BFd$ as 
pullbacks from $\RKI$.
\end{rems}

Now recall Corollary \ref{fibrns} and consider the bundle  
\[
O_\bK(2)/F^d\lra\BFd\stackrel{p_5}{\lra}BO_\bK(2),
\]
whose projection $p_5$ factorises as $p_2\scirc p_3$. The pullback 
$\gamma\letbe p_5^*(\omega_2)$ is a complex or quaternionic $2$-plane 
bundle, and is induced by a representation of $F^d$ on 
$\bK^2$, given by Definitions \ref{ndzmtssubgps}. In terms of the 
basis $\{(1,\pm1)^t\}$, the representation is equivalent to a sum 
$\alpha\oplus\beta$ of $1$-dimensionals, where 
\[
\alpha(\tau,a)\;=\,{}\cdot a\spandsp
\beta(\tau,a)\;=\,{}\cdot(-a)
\]
respectively, for any $a$ in $S^{d-1}$. In other words, there is an
isomorphism 
\begin{equation}\label{pot}
\gamma\;\cong\;
\omega\oplus(\eta\otimes_{\bR}\1\omega).
\end{equation}
\begin{lem}\label{chclpbsig} 
The characteristic classes of $\gamma$ are given by
$l_1(\gamma)=c^{d/2}+2z$ and $l_2(\gamma)=(c^{d/2}+z)z$ in
$H^*(\BFd)$, and $w_d(\gamma)=a^d$ and $w_{2d}(\gamma)=(a^d+z)z$ 
in $H^*(\BFd;\zmt)$. 
\end{lem}
\begin{proof}
By \eqref{pot}, the total Chern or symplectic Pontryagin class is 
given by
\[
l(\gamma)\;=\;(1+z)(1+c^{d/2}+z)\;=\;1+c^{d/2}+2z+c^{d/2}z+z^2.
\]
This makes $l_1$ and $l_2$ explicit, and determines 
$w_d$ and $w_{2d}$ by applying $\rho$.
\end{proof}

Also, consider the real line bundle $p_3^*(\lambda)$ over 
$\BFd$. By analogy with Remark \ref{acnonnilp}, the composition 
$C_2<F^d<P^d$ is the standard inclusion; so there is an isomorphism 
$p_3^*(\lambda)\cong\eta$, whence
\begin{equation}\label{ptlisl}
w_1(p_3^*(\lambda))\;=\;a\sands c_1(p_3^*(\lambda_\bC))\;=\;c
\end{equation}
in $H^1(\BFd;\zmt)$ and $H^2(\BFd)$ respectively.

\begin{lem}\label{p3action}
The homomorphism $p_3^*\colon Z^*\to\bZ\om[c,z]/(2c)$ 
is determined by 
\begin{equation*}
p_3^*(m)=2z,\quad p_3^*(y)=(c^{d/2}+z)z,\sands p_3^*(c)=c. 
\end{equation*}
\end{lem}
\begin{proof}
It suffices to combine Lemma \ref{chclpbsig} with \eqref{ptlisl}.
\end{proof}
Over $\zmt$, the formulae become $p_3^*(m)=0$, $p_3^*(y)=(a^d+z)z$
and $p_3^*(a)=a$. 

%
%
%
%
%
%
%
%
%

\section{Mod $2$ cohomology}\label{motwco}

In this section, the homotopy commutative ladder \eqref{cofibsi} is
exploited to compute the cohomology ring $H^*(\spt;\zmt)$ in terms of 
$H^*(B;\zmt)$ and the Thom isomorphism. Some of the results are 
specific cases of those of Nakaoka \cite{nak:cpf}. 

To ease the calculations, it is convenient to identify the 
upper cofibre sequence with \eqref{bfbpmp}, and use Remarks 
\ref{importacxy}. The long exact sequence 
\begin{equation}\label{pregysemt}
\ldots\stackrel{\delta}{\lla}H^*(RK^\infty_\ssp;\zmt)\stackrel{i^*}{\lla}
H^*(B_{\ssp};\zmt)\stackrel{\:b^*}{\lla}
H^*(\MPd;\zmt)\stackrel{\delta}{\lla}\dots
\end{equation}
becomes the Gysin sequence for the bundle $\theta$, and may be 
rewritten as
\begin{equation}\label{gyse}
\ldots\stackrel{\delta}{\lla}\zmt\om[a,z]\stackrel{p_3^*}{\lla}
G^*\stackrel{\:{}\cdot m}{\lla}G^{*-d}\stackrel{\delta}{\lla}\dots,
\end{equation}
where $\cdt m$ denotes multiplication by the Euler class 
$m=w_d(\theta)$. The Thom isomorphism identifies 
$H^*(\MPd;\zmt)$ with the free $G^*$-module $G^*\br{t}$ on a 
single $d$-dimensional generator $t$, which is known as 
the Thom class and satisfies $b^*(t)=m$. The multiplicative 
structure of $H^*(\MPd;\zmt)$ is determined by the relation 
$t^2=m\om t$, and may be encoded as an isomorphism
\begin{equation}\label{cohthsp}
H^*(\MPd;\zmt)\;\cong\;t\om G^*[t]/(t^2\?+m\om t)
\end{equation}
of algebras over $\zmt$, where $t\om G^*[t]$ represents the ideal 
$(t)\subset G^*[t]$. The implication that $p_3^*(m)=0$ is confirmed 
by the mod $2$ version of Lemma \ref{p3action}. 

So the key to calculation is held by $\delta$, and its values on 
$z^k$ for $k>0$. For notational simplicity, it is convenient to let 
$R=\bZ$ or $\zmt$, and write 
\begin{equation}\label{defdel}
\delta_k\;\letbe\;\delta(z^k)
\end{equation}
in $H^{kd+1}(\MPd;R)$. The Thom isomorphism is defined by relative cup
product, and the following basic property \cite[(8.13)]{dol:lealto} is
required.
\begin{lem}\label{relcupdel}
For any $u$ in $H^*(\BPd;R)$ and $v$ in $H^*(\BFd;R)$, the 
equation 
\[
\delta(p_3^*(u)\om v)\;=\;u\opfm\delta(v)
\]
holds in $H^*(\MPd;R)$.\hfill$\Box$
\end{lem}

Lemma \ref{p3action} and \eqref{cohthsp} then imply that 
$\delta(a^jz^k)=a^j\delta_k$ for any $j\geq 0$ and $k\geq 1$, and 
that the second order recurrence relation 
\begin{equation}\label{recur} 
\delta_{k+2}\;=\;a^d\delta_{k+1}+y\om\delta_k\sts{with}
\delta_0=0,\;\delta_1=a\om t
\end{equation}
holds in $H^*(\MPd;\zmt)$.  So $\delta$  may be found in 
terms of $a, y$ and $t$ by standard techniques from the theory of 
\emph{generalised Fibonacci polynomials} \cite{acms:gfp}.

Over any commutative ring $Q$ with identity, these polynomials 
lie in $Q[x_1,x_2]$, and are specified by the recurrence relation
\begin{equation}\label{ff}
q_{k+2}=x_1\om q_{k+1}+x_2\om q_k\sts{with} 
q_0=0,\;q_1=1;
\end{equation}
when $x_2=1$, they reduce to the Fibonacci polynomials 
\cite{jac:fpk}. Alternative choices of $q_0$ and $q_1$ create 
new sequences of polynomials, which may be written as sums 
of monomials by adapting the methods of \cite[\S2]{acms:gfp}. 
When applied to \eqref{recur} with $Q=H^*(\MPd;\zmt)$, they 
lead to the following.
\begin{prop}\label{valdelk} 
For any $k>0$, the equation 
\begin{equation}\label{dellam}
\delta_k\;=\;
\sum_{0\leq i\leq(k-1)/2}\!\binom{k-1-i}{i}a^{(k-1-2i)d+1}y^i\om t
\end{equation}
holds in $H^{kd+1}(\MPd;\zmt)$.\hfill$\Box$
\end{prop}
Equation \eqref{dellam} may also be read off from 
\cite[p252]{gou:hfq}. Only the parity of the binomial coefficients 
is relevant, and by \cite[Lemma 2.6]{step:co} $\binom{a}{b}$ is 
odd precisely when the $1$s in the dyadic expansion of $b$ form 
a subset of the $1$s in the dyadic expansion of $a$. For example
$\delta_2=a^{d+1}\om t$ and 
$\delta_3=a^{2d+1}\om t+ay\om t$.
\begin{rem}
An immediate consequences of Proposition \ref{valdelk} is that  
\[
\delta_k\;\equiv\;a^{(k-1)d+1}\om t\;\;\text{mod $(y)t$} 
\]
for all $k>0$, and therefore that $\delta_k$ is non-zero. 
\end{rem}

These results may be used to understand the long exact cohomology 
sequence of the lower cofiber sequence of \eqref{cofibsi}, for which 
the commutative square
\[
\begin{CD}
H^*(\MPd;\zmt)@<{\delta}<<H^{*-1}(\RKI;\zmt)\\[-.1em]
@A{id}AA@AA{\pi_2^*}A\\[-.1em]
H^*(\MPd;\zmt)@<{\delta}<<H^{*-1}(\KPI;\zmt)
\end{CD}
\]
is crucial. The composition $\delta\scirc\pi_2^*$ is monic, because 
$\pi_2^*\colon\zmt\om[z]\to\zmt\om[a,z]$ is the natural inclusion 
and $\delta_k$ is never $0$. The lower $\delta$ is therefore also 
monic, and Nakaoka's short exact sequence \cite[p 668]{nak:cpf} 
\begin{equation}\label{nakses}
0\lla H^*(SP^2;\zmt)\stackrel{b_\varDelta^*}{\lla}
H^*(\MPd;\zmt)\stackrel{\delta}{\lla}H^{*-1}(\KPI;\zmt)
\lla 0
\end{equation} 
emerges immediately. Of course, \eqref{nakses} splits as vector spaces 
over $\zmt$.
\begin{thm}\label{cohsptmod2}
There is an isomorphism
\[
H^*(\spt;\zmt)\;\cong\;t\om G^*[t]/(t^2\? +mt,\,\delta_k:k>0),
\]
where $t$ and $m$ have dimension $d$, and $\delta_k$ has 
dimension $kd+1$.
\end{thm}
\begin{proof}
It follows from \eqref{nakses} that there is an isomorphism
\[
H^*(\spt;\zmt)\;\cong\;H^*(\MPd;\zmt)/\delta(H^{*-1}(\KPI;\zmt))
\]
of rings, so it suffices to note that $H^*(\KPI;\zmt)$ is 
generated by the $z^k$.
\end{proof}

\begin{rems}\label{mod2prods}
Since $\delta_k\om t=0$ in $H^*(\MPd;\zmt)$ for every $k>0$, the 
summand $\bZ\br{\delta_k}$ coincides with the ideal $(\delta_k)$.
Only monomials of the form $m^iy^j\om t$ multiply non-trivially 
in $H^*(\spt;\zmt)$, subject to the relation $t^2=mt$.
\end{rems}

The isomorphism of Theorem \ref{cohsptmod2} may be clarified by
importing the values of $\delta_k$ from Proposition
\ref{valdelk}. For example, $\delta_3=0$ gives $a^5\om t=ay\om t$
in the complex case and $a^9\om t=ay\om t$ in the quaternionic case,
in dimensions $7$ and $13$ respectively.

%
%
%
%
%
%
%
%
%

\section{Integral cohomology}\label{inco}

In this section, the commutative ladder \eqref{cofibsi} is 
exploited to compute the integral cohomology rings $H^*(\MPd)$ and
$H^*(\spt)$. Input is provided by the geometric and mod $2$
cohomology results of \S\S \ref{orcose}, \ref{chcl}, and
\ref{motwco}. As in earlier sections, integral cohomology classes and
their mod $2$ reductions may be written identically when the context is 
sufficient to distinguish between them.

The integral version of \eqref{pregysemt} is the long exact sequence
\begin{equation}\label{lad1}
\ldots\stackrel{\delta}{\lla}H^*(RK^\infty_\ssp)\stackrel{\;i^*}{\lla}
H^*(B_\ssp)\stackrel{\;b^*}{\lla}
H^*(\MPd)\stackrel{\delta}{\lla}\dots.
\end{equation} 
The bundle $\theta$ has no integral Euler class, because
$w_1(\theta)=a$ by \eqref{stwhthetad}; so \eqref{gyse} has no analogue
over $\bZ$. Nevertheless, almost all the required information can be
deduced directly from \eqref{lad1}, by separating the cohomology
groups into even and odd dimensional summands.

By Theorem \ref{cohbpd} and Remarks \ref{importacxy}, 
$H^{od}(B)$ and $H^{od}(\RKI)$ are zero. So \eqref{lad1} reduces 
to a collection of exact sequences of length $4$, which together 
imply that $b^*$ and $\delta$ induce isomorphisms
\begin{equation}\label{intisos}
H^{ev}(\MPd)\;\cong\;\Ker i^*\spandsp
H^{od}(\MPd)\;\cong\;\Cok i^*
\end{equation}
respectively. The relative cup product invests $H^*(\MPd)$ with 
a natural module structure over $Z^*$, and the first isomorphism is 
of $Z^*$-algebras; it interacts with $\delta$ as in Lemma 
\ref{relcupdel}. The second isomorphism is of graded abelian groups,
and shifts dimension by $1$. 

\begin{lem}\label{kerp3}
The kernel of $i^*$ is the principal ideal $(m^2-4y)$ in $Z^*$.
\end{lem}
\begin{proof}
Lemma \ref{p3action} implies that $i^*(m^2-4y)=0$, so 
$(m^2-4y)\subseteq\Ker i^*$. 

On the other hand, let $f(m,y)+g(c,y)$ denote an arbitrary 
element 
\begin{equation}\label{fgformat}
\sum_{i=0}^kf_i\, m^{2(k-i)}y^i\;+\;\:
\sum_{j=0}^{k-1}g_j\, c^{(k-j)d}y^j
\end{equation}
of $Z^{2kd}$, and suppose that it is annihilated by $i^*$; then 
\[
f(2z,(c^{d/2}+z)z)+g(c,(c^{d/2}+z)z)\;=\;0
\]
in $\bZ\om[\om c,z]/(2c)$. Equating coefficients of $c^{kd}$,
$c^{(k-1/2)d}z$, \dots, $c^{(k-(k-1)/2)d}z^{k-1}$ shows
that $g_0$, $g_1$, \dots, $g_{k-1}$ $\equiv 0$ mod $2$, and hence 
that $g(c,y)=0$. Equating coefficients of $z^{2k}$ shows that
$\fssp\letbe\sum_{i=0}^k2^{2(k-i)}f_i=0$, and hence that
\[
f(m,y)|_{m^2=4y}\;=\;\fssp y^n\;=\;0. 
\]
Thus $m^2-4y$ divides $f(m,y)$. Similar arguments apply to 
elements of $Z^{2(kd+q)}$ for $1\leq q<d$, and confirm that 
$\Ker i^*\subseteq(m^2-4y)$ in $Z^*$. 
\end{proof}

\begin{rem}\label{csiszero}
To describe the image of $i^*$, it therefore suffices to work 
modulo $(m^2-4y)$; moreover, every element of the ideal 
takes the form $f(m,y)(m^2-4y)$ for some polynomial $f$ in 
$\bZ\om[m,y]$, because $c(m^2-4y)=0$ in $Z^{2d+2}$.
\end{rem}

\begin{lem}\label{cokp3}
The cokernel of $i^*$ is isomorphic to the $\zmt$-module on 
basis elements $c^iz^j$ for $0\leq i<jd/2$.
\end{lem}
\begin{proof}
First observe that $i^*(2y^k)=2z^{2k}$ and $i^*(my^k)=2z^{2k+1}$ in
$H^*(RK^\infty_{\ssp})$, for every $k>0$. So $\Cok i^*$ is a
$\zmt$-module, spanned by those elements $c^iz^j$ for which $i\geq 0$ 
and $j>0$. There are relations amongst them because
\begin{equation}\label{relelk}
i^*(y^j)\;=\;(c^{d/2}+z)^jz^j\;=\;0,
\end{equation}
in $\Cok i^*$, and \eqref{relelk} may also be multiplied by any $c^i$. 
The resulting expression
\begin{equation}\label{czaccept}
c^{jd/2}z^j\;=\;\sum_{0<i\leq j}\binom{j}{i}c^{(j-i)d/2}z^{j+i}
\end{equation}
describes $c^{jd/2}z^j$ as a homogeneous linear combination of
monomials $c^az^b$ for which $0\leq a<bd/2$. By iteration, a similar
description arises for \emph{any} \smash{$c^iz^j$}; this iteration
terminates, because every step reduces the powers of $c$ that occur.

Linear independence of the $c^az^b$ follows by assuming that
\[
\sum_{0<i\leq k}\epsilon_i\opfm c^{(k-i)d/2}z^{k+i}\;=\;
f(2z,(c^{d/2}+z)z)+g(c,(c^{d/2}+z)z)
\]
in $H^{2kd}(\RKI)$, with $\epsilon_i=0$ or $1$ for $i<k$. Adapting 
the  proof of Lemma \ref{kerp3} shows that  
\mbox{$g_0$, $g_1$, \dots, $g_{k-1}$, $f_k\equiv 0$ mod $2$}, 
so $\epsilon_i=0$ for $i<k$ and $\epsilon_k\equiv 0$ mod $2$, as 
required. Similar arguments apply in dimensions $2(kd+q)$, for 
$1\leq q<d$.
\end{proof}
To illustrate the iteration, note that \eqref{czaccept} gives
$c^{3d/2}z^3=c^d z^4+c^{d/2}z^5+z^6$ for $j=3$; so
$c^{5d/2}z^3=c^{2d}z^4+c^{3d/2}z^5+c^dz^6$. Importing $c^{2d}z^4=z^8$
yields
\[
c^{5d/2}z^3\;=\;c^{3d/2}z^5+c^dz^6+z^8,
\] 
which is of the required format.

\begin{rem}\label{redmonic}
Lemma \ref{cokp3} shows that the coefficient homomorphism 
$2\colon\bZ\to\bZ$ induces the zero homomorphism on 
$H^{od}(\MPd)$, and the associated long exact cohomology sequence 
proves that mod $2$ reduction is monic.
\end{rem}

\begin{thm}\label{intcohmpd}
As a $Z^*$-module, $H^*(\MPd)$ is generated by a $2d$-dimensional 
element $s$ of infinite order, and the $jd+1$-dimensional 
elements $\delta_j$ of order $2$, for $j>0$. The 
module structure is determined by $c\om s=m\om \delta_j=0$,
\[
c^{jd/2}\om\delta_j=
\!\sum_{0<i\leq j}\binom{j}{i}c^{(j-i)d/2}\delta_{j+i},\sands\;
y\opfm\delta_j=c^{d/2}\delta_{j+1}+\delta_{j+2};
\]  
the algebra structure is determined by
\[
s^2\:=\:(m^2-4y)\om s\sands 
\delta_i\delta_j\:=\:s\delta_i\;=\;0,
\]
for every $i>0$. 
\end{thm}
\begin{proof}
Define $s$ by $b^*(s)=m^2-4y$, and $\delta_j$ by \eqref{defdel}. 
The abelian group structure then follows from Lemmas \ref{kerp3} and 
\ref{cokp3}, and Remark \ref{csiszero}. The relation $c\om s=0$ holds 
because $b^*(c\om s)=c(m^2-4y)=0$ by Theorem \ref{cohbpd}, and 
$b^*$ is monic; also, $m\om\delta_j=\delta(2z^{j+1})=0$ for any 
$j$, by Lemma \ref{relcupdel}. The formula for $c^{jd/2}\om\delta_j$ 
arises by applying $\delta$ to \eqref{czaccept}, and the formula for 
$y\opfm\delta_j$ is the integral analogue of \eqref{recur}.

The relation $s^2=(m^2-4y)\om s$ holds because $b^*$ is monic, and
$\delta_i\delta_j=0$ because $H^{ev}(\MPd)$ is torsion free. Finally,
the mod $2$ reduction $\rho(s\delta_i)=\rho(s)\om\delta_i$
is zero by Remarks \ref{mod2prods}, so $s\delta_i=0$, by Remark
\ref{redmonic}.
\end{proof} 
\begin{cor}\label{preuij}
There is an isomorphism
\[
H^{od}(\MPd)\;\cong\;\zmt\om\br{c^i\delta_j:0\leq i<jd/2}
\]
of graded abelian groups.
\end{cor}
By analogy with the mod $2$ case, Theorem \ref{intcohmpd} may now 
be applied to study the lower row of the cohomology ladder  
\[
\begin{CD}
\dots\!@<{\delta}<<\! H^*(RK^\infty_\ssp)@<{i^*}<<H^*(B_\ssp)
@<{b^*}<<\,H^*(\MPd)@<{\delta}<<\!\dots\\[-.1em]
@.\!\!@AA{\pi\smash{_2^*}}A@AA{\pi^*}A@AA{id}A\\[-.1em]
\dots\!@<{\delta}<<\! H^*(\bK P^\infty_\ssp)@<{i_\varDelta^*}<<H^*(\sptssp)
@<{b_\varDelta^*}<<\,H^*(\MPd)@<{\delta}<<\!\dots
\end{CD}
\]
associated to \eqref{cofibsi}.  

\begin{rems}\label{remslad}
Since $\pi_2^*\colon\bZ\om[z]\to\bZ\om[c,z]/(2c)$ is the 
canonical inclusion, Theorem \ref{intcohmpd} shows that the
lower $\delta$ acts by $\delta(z^j)=\delta_j$ in $H^{kd+1}(\MPd)$
for every $j>0$. Since $2\delta_j=0$, the kernel of $\delta$ is 
the principal ideal $(2z)$, and is additively generated by $2z^j$ 
in each dimension $jd$.
\end{rems}

These observations lead to additive descriptions of $H^{od}(\spt)$ 
and $H^{ev}(\spt)$.
\begin{prop}\label{odcohspt}
There are isomorphisms
\[
H^{od}(\spt)\;\cong\;\zmt\om\br{u_{i,j}:0<i<jd/2}
\]
of graded abelian groups, where $u_{i,j}$ has dimension $2i+jd+1$;
in particular, $H^{od}(\spt)=0$ only in dimensions 
$1,\om3$ and $5$, and also $9$ when $\om\bK=\bH$.  
\end{prop} 
\begin{proof}
Remarks \ref{remslad} imply that $b^*_\varDelta(\delta_j)=0$ in 
$H^{jd+1}(\spt)$, and that 
\begin{equation}\label{defuij}
u_{i,j}\;\letbe\;b^*_\varDelta(c^i\delta_j)
\end{equation}
is nonzero and has order $2$ in $H^{2i+jd+1}(\spt)$, for every 
$0<i<jd/2$. By Corollary \ref{preuij}, these elements exhaust the 
image of $b^*_\varDelta$, which is epic in odd dimensions because 
$H^{od}(\KPI)$ is zero. 
\end{proof}
Note that $u_{jd/2\om ,\om j}=\sum_{0<i\leq j}\binom{j}{i}
u_{(j-i)d/2\om ,\om j+i}$ 
in $H^{2jd+1}(\spt)$, for any $j>0$. 
\begin{rem}\label{redmonic2}
Proposition \ref{odcohspt} shows that mod $2$ reduction is monic 
on $H^{od}(\spt)$, by analogy with Remark \ref{redmonic}.
\end{rem}

Remarks \ref{remslad} confirm the existence of elements $n_p$ in 
$H^{pd}(\spt)$ such that $i_\varDelta^*(n_p)=2z^p$, for every $p>0$;
they are well-defined up to the image of $b^*_\varDelta$. 

\begin{prop}\label{evdimfree}
The graded abelian group $H^{ev}(\spt)$ is torsion-free, and admits 
a canonical choice of generator $n_p$ in each dimension $pd$. 
\end{prop}
\begin{proof}
By Lemma \ref{kerp3}, $b^*$ injects $H^{ev}(\MPd)$ into 
$H^{ev}(B)$ as the ideal $(m^2-4y)$. Thus $b^*_\varDelta$ 
injects $H^{ev}(\MPd)$ as a subring $Q^*$ of $H^{ev}(\spt)$, 
and $\pi^*$ maps $Q^*$ isomorphically to $(m^2-4y)$. So 
there is a split short exact sequence
\begin{equation}\label{sesevdim}
0\llla\bZ\br{2z^p}\stackrel{i^*_\varDelta}{\llla}\bZ\br{n_p}\oplus
Q^{pd} \stackrel{b^*_\varDelta}{\llla}\bZ\br{m^qy^r\? s}\llla 0
\end{equation}
for each $p\geq 1$, where $q$ and $r$ range over $q,\, r\geq 0$
such that $q+2r+2=p$. As $p$ varies, \eqref{sesevdim} confirms 
that $H^{ev}(\spt)$ is torsion free.

Since $\pi^*$ is a $\bZ[1/2]$ isomorphism, it is monic on
$H^{ev}(\spt)$. In dimension $2kd$, therefore, $n_{2k}$ may be
specified uniquely by $\pi^*(n_{2k})=2y^k$. For, if
no such element exists, then $\pi^*(n_{2k})-2y^k=\alpha$ for some
nonzero $\alpha$ in $(m^2-4y)$; hence $\pi^*(\alpha')=\alpha$ for some
$\alpha'$ in $Q^{2kd}$, and $\pi^*(n_{2k}-\alpha')=2y^k$, which is a
contradiction.  The same argument works in dimension $(2k-1)d$, after
replacing $2y^k$ by $my^{k-1}$. In either dimension $n_p$ is an
additive generator.
\end{proof}

Henceforth, $n_p$ is assumed to be chosen in this way, for any
$p$. The cases
\begin{equation}\label{defgh}
g\;\letbe\; n_1\sands 
h\;\letbe\; n_2
\end{equation}
in $H^d(\spt)$ and $H^{2d}(\spt)$ are sufficiently important to merit 
special notation.
\begin{thm}\label{intcohspt}
There is an isomorphism
\[
H^*(\sptssp)\;\cong\;
\bZ\opfm[\om h^p\!/2^{p-1}\!,\,g^qh^r\!/2^r\!,\,u_{i,j}]\,/\,I,
\]
where $p,\,q\geq 1$, $r\geq 0$ and $0<i<jd/2$, and $I$ denotes the 
ideal
\[
(2u_{i,j},\,u_{i,j}u_{k,l},\,u_{i,j}h^p\!/2^{p-1}\!,\,
u_{i,j}g^qh^r\!/2^r)\,;
\]
the classes $g$, $h$, and $u_{i,j}$ are those of \eqref{defgh} and 
\eqref{defuij} respectively.
\end{thm}
\begin{proof}
Since $\pi^*$ embeds $H^{ev}(\spt)$ in
$H^{ev}(B;\bZ[1/2])$, the canonical additive generators $g$ and $h$
may be identified with their images $m$ and $2y$ respectively;
similarly, the generators $n_{2k-1}$ and $n_{2k}$ may be identified
with $my^{k-1}$ and $2y^k$, and therefore with $gh^{k-1}/2^{k-1}$
and $h^k/2^{k-1}$. In particular, $b^*_\varDelta(s)=g^2-2h$.

From this viewpoint, the short exact sequence \eqref{sesevdim}
prescribes additive bases
\[
h^k\!/2^{k-1}\!,\,h^{k-1}(g^2\?-\?2h)/2^{k-1}\!,\,\dots,\,
g^{2k-2r-2}h^r(g^2\?-\?2h)/2^r,\,\dots,\,g^{2k-2}(g^2\?-\?2h)
\]
in dimensions $2kd$, and
\[
gh^k\!/2^k\!,\,gh^{k-1}(g^2\?-\?2h)/2^{k-1}\!,\,\dots,\,
g^{2k-2r-1}h^r(g^2\?-\?2h)/2^r,\,\dots,\,g^{2k-1}(g^2\?-\?2h)
\]
in dimensions $(2k+1)d$, where $0\leq r<k$. These are equivalent 
to bases
\[
h^k\!/2^{k-1}\!,\dots,
g^{2k-2r}h^r\!/2^r\!,\dots,g^{2k}\!\sands\!
gh^k\!/2^k\!,\dots,g^{2k-2r+1}h^r\!/2^r\!,\dots,g^{2k+1}
\]
respectively, which exhibit the stated multiplicative structure on 
$H^{ev}(\spt)$.

For the products in $I$, note that $u_{i,j}u_{k,l}=0$ because
$H^{ev}(\spt)$ is torsion free, by Proposition \ref{evdimfree}.
Furthermore, $\rho(u_{i,j})=b_{\varDelta}^*(a^{2i}\delta_j)$ in
$H^{2i+jd+1}(\spt;\zmt)$, so
$\rho(u_{i,j}n_p)=\rho(u_{i,j}g^qh^r/2^r)=0$ by Remarks
\ref{mod2prods}. Finally, Remark \ref{redmonic2} shows that
$u_{i,j}n_p=u_{i,j}g^qh^r/2^r=0$, as required.
\end{proof}

%
%
%
%
%
%
%
%
%

\section{Mod 2 truncation}\label{motwtr}

In the final two sections, it remains to describe how the cohomology rings 
of $\varGamma_n$, $B_n$, and $\sptn$ are obtained from the calculations 
above, using the restriction homomorphisms induced by inclusion. All the
resulting truncations are valid for any $n\geq 1$, although the case $n=1$ 
is degenerate. For notational simplicity, cohomology classes and their 
restrictions are written identically in both sections. 

This section focuses mainly on mod $2$ cohomology.

The first space to consider is $\varGamma_n$. As described in
Definition \ref{defgamman}, it is the total space 
$\bR P(\chi_n)$ of the projectivisation of the restriction to
$Gr_{n+1,2}$ of the universal $Spin(d+1)$ bundle over
$BO_{\bK}(2)$. By Corollary \ref{dadhalg}, it is also homotopy 
equivalent to the configuration space $\mathcal{C}_2(\KPn)$, whose 
one-point compactification is the Thom space $\Th(\theta_n)$.

The cohomology ring of the Grassmannian is given by the truncation  
\[
H^*((Gr_{n+1,2})_{\ssp};R)\;\cong\;R[l_1,l_2]/
(\sigma_n,\sigma_{n+1})
\]
of $H^*(BO_{\bK}(2);R)$, where $R=\bZ$ or $\zmt$ and 
the $\sigma_k$ satisfy
\begin{equation}\label{ffsigtil}
\sigma_{k+2}=-l_1\sigma_{k+1}-l_2\sigma_k
\sts{with}\sigma_0=1,\;\sigma_1=-l_1
\end{equation}
in $R[l_1,l_2]$. Thus \eqref{ffsigtil} fits into the framework of 
\eqref{ff}, and leads to 
\begin{equation}\label{yaspoly}
\sigma_k\;=\;
\sum_{0\leq i\leq k/2}(-1)^{k-i}\binom{k-i}{i}l_1^{k-2i}l_2^i
\end{equation}
in dimension $kd$, for $k\geq 1$; this agrees with Yasui 
\cite[(4.7)]{yas:rsp}. Proceeding as for \eqref{isobst}, 
$H^*((\varGamma_n)_{\ssp};\zmt)$ may then be identified with
the $\zmt$-algebra
\[
G^*_n\;\letbe\;\zmt\om[\om a,m,y]/(am,\nu_n,\nu_{n+1}),
\]
where $\nu_k=\sigma_k(a^d\2+\2 m,\om y)$
in $H^{kd}(\varGamma;\zmt)$, for $k\geq 1$.

The analogues of \eqref{gyse} and \eqref{cohthsp} confirm the 
existence of an isomorphism 
\begin{equation}\label{cohtruncthsp}
H^*(\Th(\theta_n);\zmt)\;\cong\;t\om G^*_n[t]/(t^2\?+m\om t),
\end{equation}
where $t$ is pulled back from the universal Thom class 
in $H^d(\MPd;\zmt)$. The associated Euler class is $m$, as before. 
\begin{rem}\label{cohmpdtrunc}
Care is required to work with the restriction homomorphism on
$H^*(\MPd;\zmt)$. Its kernel is the ideal $(\nu_n,\nu_{n+1})\om t$, so 
\eqref{cohtruncthsp} may be rewritten as
\[
H^*(\Th(\theta_n);\zmt)\;\cong\;
H^*(\MPd;\zmt)/(\nu_n,\nu_{n+1})\om t
\]
for any $n\geq 1$. For example, $\nu_3=x^3$ in $G_2^{3d}$, so 
$a^2x^3\om t$ lies in $(\nu_3)\om t$ and is $0$ in 
$H^{4d+2}(\Th(\theta_3);\zmt)$; however, $a^2x^3\om t$ does 
not lie in the ideal $(\nu_3\om t)$.
\end{rem}

In order to complete the calculation of $H^*(\sptn;\zmt)$, it suffices 
to apply restriction to \eqref{nakses}, and consider the resulting exact 
sequence
\begin{equation}\label{nakladtrunc}
\dots\!\lla\!H^*(\sptn;\zmt)\!\stackrel{b_\varDelta^*}{\lla}\!
H^*(\Th(\theta_n);\zmt)\!\stackrel{\delta}{\lla}\!H^{*-1}(\KPn;\zmt)
\!\lla\! \dots.
\end{equation} 
Of course, Nakaoka's results show that \eqref{nakladtrunc} is also
short exact. In the current context, this follows from the fact that 
restriction truncates $H^*(\KPI;\zmt)$ and $H^*(\MPd;\zmt)$ by 
the ideals $(z^{n+1})$ and $(\nu_n,\om\nu_{n+1})\om t$ 
respectively. So restriction is an isomorphism on 
$H^*(\MPd;\zmt)$ in dimensions 
$\leq(n+1)d-1$, and maps $\delta_k$ to its non-zero namesake in 
$H^{kd+1}(\Th(\theta_n);\zmt)$ for each $1\leq k\leq n$. Hence 
$\delta$ in \eqref{nakladtrunc} is monic, as required.
\begin{thm}\label{cohsptnmod2}
For any $n\geq 1$, there are isomorphisms
\begin{align*}
H^*(\sptn;\zmt)&\;\cong\;
t\om G_n^*[t]/(t^2\? +m\om t,\,\delta_k:0<k\leq n)\\
&\;\cong\;H^*(\spt;\zmt)\big/
(\nu_n,\nu_{n+1})\om t/(a\nu_{n+l}\om t:l\geq 0)
\end{align*}
of graded $\zmt$-algebras.
\end{thm}
\begin{proof}
The first isomorphism arises by adapting the proof of Theorem 
\ref{cohsptmod2} to the short exact sequence \eqref{nakladtrunc}. 
In particular, restriction induces an epimorphism onto 
$H^*(\sptn;\zmt)$, with kernel 
$b_{\varDelta}^*((\nu_n,\nu_{n+1})\om t)$. Comparing \eqref{dellam} 
with \eqref{yaspoly} and using the relation $a\om x=a^{d+1}$ from 
\eqref{stwhthetad}, gives $\delta_{r+1}=a\nu_r\om t$ in 
$H^{(r+1)d+1}(\MPd;\zmt)$ for any $r\geq 1$; 
so $b_{\varDelta}^*$ has kernel $(a\nu_{n+l}\om t:l\geq 0)$, and the 
second isomorphism holds.
\end{proof}
For example, the classes $\nu_2\om t=(a^{2d}+m^2+y)\om t$ and
$\nu_3\om t=(a^{3d}+m^3)\om t$ are nonzero in 
$(\nu_2,\nu_3)\om t/(a\nu_{2+l}\om t:l\geq 0)$. In general, 
$(a\nu_{n+l}\om t:l\geq 0)$ is simply a copy of $\zmt<
(\nu_n,\nu_{n+1})\om t$ in each dimension $(n+l+1)d+1$.

%
%
%
%
%
%
%
%
%

\section{Integral truncation}\label{intr}

Our final section completes the calculations, by focusing on the
integral cohomology rings of $\varGamma_n$, $B_n$, and $\sptn$.  An
important intermediate step is the study of $H^*(\Th(\theta_n))$,
whose global structure is best described in the spirit of
\cite{cad:cbw}, using local coefficients. That approach will,
however, be pursued elsewhere, and attention will be restricted here
to the applications.

As in the mod $2$ case, the first space to consider is the
configuration space $\varGamma_n\simeq \mathcal{C}_2(\KPn)$, 
whose integral
cohomology follows from the Leray-Serre spectral sequence for the
fibration $\bR P^d\to\varGamma_n\to Gr_{n+1,2}$, by restricting
the cohomology of the base in the proof of Theorem \ref{cohbpd}.
\begin{thm}\label{intcohbpdn}
For any $n\geq 1$, there is an isomorphism 
\[
H^*((\varGamma_n)_{\ssp})\;\cong\;
\bZ\om[c,m,y]/(2c,\om cm,\nu_n,\nu_{n+1})\;\cong\;
H^*(\varGamma_{\ssp})/(\nu_n,\nu_{n+1})
\]
of graded rings, where $\nu_k=\sigma_k(c^{d/2}\2+\2 m,\om y)$ 
in $H^{kd}(\varGamma)$ for every $k\geq 1$. 
\end{thm}

Our next task is to study $H^*(\Th(\theta_n))$. The crucial geometric
input is the commutative diagram \eqref{thomlad} (whose upper row no
longer arises from a $3$-sphere bundle for any finite $n$). Applying
$H^*(-)$ gives the commutative ladder
\begin{equation}\label{intcohladn}
\begin{CD}
\dots\!\!@<{\delta}<<\!\! H^*(RK^n_{\ssp})\?
@<{i_n^*}<<\!\?\,H^*((B_n)_{\ssp})\?\!@<{b_n^*}<<\?\,H^*(\Th(\theta_n))\!\!
@<{\delta}<<\!\dots\\[-.1em]
@.\!@A{\pi\smash{_2^*}}AA\;@AA{\pi^*}A@AA{id}A\\[-.1em]
\dots\!\!@<{\delta}<<\!\! H^*(\bK P^n_{\ssp})\?@<{i_\varDelta^*}<<
\!\?\,H^*((\sptn)_{\ssp})\?\!@<{b_\varDelta^*}<<\?\,H^*(\Th(\theta_n))
@<{\delta}<<\!\dots
\end{CD}
\end{equation}
for any $n\geq 1$. Elements $c,\,x,\,m$, and $y$ are defined in
$H^*(B_n)$ via Remarks \ref{importacxy}, such that $H^*(RK^n_{\ssp})$ is the
truncation $\bZ[c,z]/(2c,z^{n+1})$ and $H^{od}(\RKn)=0$.  To describe
$H^*(B_n)$ more precisely, a brief diversion is required.

For any $k\geq 0$, consider the \emph{power sum} polynomial
\begin{equation}\label{newtpoly}
r_k(e_1,\om e_2)\;=\;z_1^k + z_2^k\,,
\end{equation}
where $e_1=z_1+z_2$ and $e_2=z_1z_2$. The $r_k$ lie in 
$\bZ\om[e_1,e_2]$, and satisfy
\begin{equation}\label{psrecur}
r_{k+2}=e_1\om r_{k+1}-e_2\om r_k\sts{with}
r_0=2,\;r_1=e_1.
\end{equation}
They fit into the framework of \eqref{ff}, and are given by
\begin{equation}\label{ffdefr}
r_k\;=\;\sum_{0\leq i\leq k/2}
(-1)^i\left(2\binom{k-i}{i}-\binom{k-i-1}{i}\right)e_1^{k-2i}e_2^i\,.
\end{equation}  

\begin{rems}\label{psprops}
Iterating \eqref{psrecur} shows that $r_{k+t}$ lies in the ideal 
$(r_k,r_{k+1})$ for any $t\geq 0$, and that 
$r_k\equiv(-1)^{k/2}2e_2^{k/2}$ mod $(e_1)$ for every even $k$. The 
expression
\begin{equation}\label{psjk}
e_2^j\om r_k\;=\;\sum_{0\leq i\leq j}(-1)^i\binom{j}{i}e_1^{j-i}r_{k+j+i}
\end{equation}
holds for any $j$, $k\geq 0$, by rewriting $e_2$ as 
$z_1(e_1-z_1)=z_2(e_1-z_2)$.
\end{rems}
The polynomials $r_k$ help to extend unpublished work of 
Roush \cite{rou:stc} on $H^*(B)$.
\begin{thm}\label{intcohbn}
The integral cohomology ring $H^*((B_n)_{\ssp})$ is 
isomorphic to
\[
Z^*_n\;\om\text{\rm $\letbe$}\;\;
\bZ\om[\om c,m,y]/(2c,\om cm,\om r_{n+1},r_{n+2},\om y^{n+1})
\;\cong\;H^*(B_{\ssp})\?/(r_{n+1},r_{n+2},\om y^{n+1})
\]
where $r_k=r_k(m,y)$ has dimension $kd$ for every $k\geq 0$.
\end{thm}
\begin{proof}
Consider the Leray-Serre spectral sequence 
\[
E^{p,q}_2\;\letbe\; H^p(\RPIssp;\mathcal{H}^q((\KPn\2\times\KPn)_{\ssp})
\;\Longrightarrow\;H^{p+q}((B_n)_{\ssp})
\]
of the Borel bundle for $B_n$, where $\mathcal{H}^*(\;)$ denotes
cohomology twisted by the involution $\iota$. Since $\iota^*$ acts on
$\bZ\om[z_1,z_2]/(z_1^{n+1},z_2^{n+1})$ by interchanging $z_1$ and
$z_2$, the ring of invariants is isomorphic to
$\bZ\om[e_1,e_2]/\Ker\mu$, where $\mu$ is the projection of
$\bZ\om[e_1,e_2]$ into $\bZ\om[z_1,z_2]/(z_1^{n+1},z_2^{n+1})$. It
follows from \eqref{psrecur} that $\Ker\mu$ is the 
ideal generated by $e_2^{n+1}$, $r_{n+1}$, and $r_{n+2}$, and hence 
that
\begin{equation*}\label{invts}
E^{0,*}_2\;=\;\,\bZ\om[e_1,e_2]/(e_2^{n+1},r_{n+1},r_{n+2}).
\end{equation*}

Moreover, there is an isomorphism $E^{*,0}_2\cong H^*(\RPIssp)$. So 
the standard co\-chain complex (as constructed in 
\cite[Chapter XII \S7]{caei:ha}, for example) leads to a 
multiplicative isomorphism
\begin{equation}\label{e2multtrunc}
E^{*,*}_2\;\cong\;
(\bZ[c]/(2c))\otimes_{\bZ}E^{0,*}_2\big/(c\otimes e_1),
\end{equation}
where $c\otimes e_1=0$ holds because $E^{2,d}_2=0$. This, in
turn, is a consequence of the equation $e_1=z_1+\iota^*(z_1)$ in
\smash{$\bZ\br{e_1}/\Ima(1+\iota^*)\cong E^{2,d}_2$.}

Every differential is zero for dimensional reasons, so
\eqref{e2multtrunc} actually describes the $E_\infty$ term.
Furthermore, $m$ and $y$ in $H^*(B_n)$ represent $e_1$ and $e_2$ in
\smash{$E^{0,*}_\infty$} respectively, and $c$ in $H^2(B_n)$ represents 
$c$ in $E^{2,0}_\infty$. Thus $c^if(m,y)$ in $H^{2i+jd}(B_n)$ represents
$c^i\otimes f(e_1,e_2)$ in $E^{2i,jd}_\infty$ for any homogeneous
polynomial $f(m,y)$; all extension problems are therefore trivial
and \eqref{e2multtrunc} is additively isomorphic to $H^*(B_n)$. Since
$cm$ represents $c\otimes e_1$ and both are zero, the isomorphism
is also multiplicative, and the result follows.
\end{proof}
\begin{cor}\label{bntrunc}
For any $n\geq 1$, the monomials $m^iy^j$ form a basis for a maximal 
torsion-free summand of $H^{ev}(B_n)$, where $1\leq i+j\leq n$; also, 
$H^{od}(B_n)=0$. 
\end{cor}
\begin{proof}
It suffices to work with $e_1$ and $e_2$ in \smash{$E^{0,*}_2$}, which is 
zero when $*$ is odd.

When $*\leq nd$ is even, the monomials $e_1^ie_2^j$ already form a
basis. In dimension $(n+1)d$ and higher, the relations take effect;
for example, $e_1^{n+1}$ is divisible by $e_2$ modulo
$\Ker\mu$. Similarly, by repeated appeal to \eqref{psjk},
$e_1^{n+1-s}e_2^s$ is divisible by $e_2^{s+1}$ modulo $\Ker\mu$ in
each dimension $(n+1+s)d$ for which $0\leq s\leq n$, and
$e_2^{n+1}=0$. These relations, together with their multiples by
powers of $e_1$, show that \smash{$E^{0,td}_2$} is spanned by those
$e_1^ie_2^j$ for which $i+2j=t$ and $i+j\leq n$, where $n<t\leq 2n$;
they form a basis because their projections are linearly independent
under $\mu$. All monomials of higher dimension lie in $\Ker\mu$.
\end{proof}
\begin{rem}\label{topcell}
The highest dimensional non-torsion elements are in $H^{2nd}(B_n)$,
where the monomial $y^n$ generates a single summand $\bZ$.    
\end{rem}

By analogy with \eqref{intisos}, the upper row of \eqref{intcohladn} 
induces isomorphisms
\begin{equation}\label{intisostrunc}
H^{ev}(\Th(\theta_n^d))\;\cong\;\Ker i_n^*\spandsp
H^{od}(\Th(\theta_n^d))\;\cong\;\Cok i_n^*,
\end{equation}
where the first is of algebras over the ring $Z_n^*$ of 
Theorem \ref{intcohbn}.  These isomorph\-isms are best discussed in
the context of the commutative ladder 
\begin{equation}\label{infcofibsn}
\begin{CD}
\dots\!\!@<{\delta}<<\!\! H^*(RK^n_{\ssp})\!\!
@<{i_n^*}<<\!\!H^*((B_n)_{\ssp})\!@<{b_n^*}<< H^*(\Th(\theta_n))
\!\!@<{\delta}<<\!\dots\\[-.1em]
@.\!\!\!@A{}AA\;@A{}AA\;@A{}AA\\[-.1em]
\dots\!\!@<{\delta}<< H^*(RK^{\infty}_{\ssp})\!
@<{\;i^*}<<H^*(B_{\ssp})@<{\;b^*}<<\!\! H^*(\MPd)\!\!
@<{\delta}<<\!\dots
\end{CD}
\end{equation}
which arises by restriction to $B_n$. 

\begin{prop}\label{kerin}
For $n\geq 2$, the kernel of $i_n^*$ is the principal ideal 
$(m^2-4y)$, and for $n=1$, it is $(2y)$; in dimensions 
$>nd$ the monomials $m^iy^j$ and $2y^k$ form an additive basis, 
where $i\geq 1$, $i+j\leq n<i+2j$ and $n/2<k\leq n$.
\end{prop}
\begin{proof}
Truncation has no effect in dimensions $\leq nd$, so Lemma \ref{kerp3} 
continues to hold; furthermore, \eqref{infcofibsn} confirms that 
$i_n^*(m^2-4y)=0$ for any $n\geq 1$.

Now let $\alpha$ be such that $i_n^*(\alpha)=0$ in $H^{(n+s)d}(\RKn)$,
for some $s\geq 1$. Restriction is epic, so $\alpha$ lifts to an
element $\alpha'$ in $H^{(n+s)d}(B)$ whose expansion \eqref{fgformat}
has $f_i=g_j=0$ for $i,j>n$. Since $i^*(\alpha')\equiv 0$ mod
$(z^{n+1})$ in $H^{(n+s)d}(\RKI)$, the proof of Lemma \ref{kerp3} may
then be modified to show that $i^*(\alpha')=2\lambda z^{n+s}$ for some
integer $\lambda$. Moreover, $i^*(\lambda\om r_{n+s})=2\lambda
z^{n+s}$ by \eqref{newtpoly}, so $\alpha'-\lambda\om r_{n+s}$ lies in
$\Ker i^*$, and therefore in $(m^2-4y)$ by Lemma \ref{kerp3}.
Restriction then confirms that $\alpha$ lies in $(m^2-4y)$, as
required. This argument also works for $n=1$ because $r_2=m^2-2y$,
so the relation $r_2=0$ gives $m^2=2y$ in $H^{2d}(B_1)$.

In terms of Corollary \ref{bntrunc}, the monomials that lie in
dimensions between $(n+1)d$ and $2nd$ clearly satisfy
$i_n^*(m^iy^j)=0$ whenever $i\geq 1$; on the other hand,
$i_n^*(y^k)\equiv c^{kd/2}z^k$ mod $(cz^{k+1})$ is of order 
$2$, but non-zero. It follows that $\Ker i_n^*$ has the stated basis 
in this range.
\end{proof}

\begin{rems}\label{rkinsp}
By Remarks \ref{psprops}, the polynomials $r_k(m,y)$ are
$\equiv(-1)^{k/2}2y^{k/2}$ or $0$ mod $(m)$ in $H^{kd}(B)$, for even or
odd $k$ respectively. So by Theorem \ref{intcohspt}, they lie in the
image of $\pi^*$, and may therefore be rewritten as $r_k(g,h)$ in
$H^{kd}(\spt)$ without ambiguity. In this context, \eqref{psjk}
confirms that every element $g^ah^jr_k/2^j$ belongs to the ideal
\smash{$R^{ev}_{k+j-1}\letbe(r_{\ell}(g,h):\ell>k+j-1)$} of
$H^{ev}(\spt)$.  Examples are {\bf(1)} $h^jr_k/2^j$ when $a=0$, $k>1$,
{\bf(2)} $g^{a+1}h^j/2^j$ when $k=1$, and {\bf(3)} $h^j/2^{j-1}$
when $a=k=0$.
\end{rems}
\begin{cor}\label{sptntf}
The cohomology groups $H^{ev}(\sptn)$ are torsion free; in dimensions
$>nd$ the monomials $g^qh^s\!/2^s$ and $h^p\!/2^{p-1}$ form an
additive basis, where $q\geq 1$, $q+s\leq n<q+2s$ and $n/2<p\leq n$.
\end{cor}
\begin{proof}
Combine Proposition \ref{kerin} with \eqref{defgh},
\eqref{intisostrunc}, and the fact that $b^*_\varDelta$ is an
isomorphism for dimensions $>nd$ in diagram \eqref{intcohladn}.
\end{proof}

To conclude the calculations, consider the commutative square 
\[
\begin{CD}
H^*(B)\!\!@>>>\!H^*(B_n)\\[-.1em]
\;@A{\pi^*}AA\!\!@AA{\pi^*}A\\[-.1em]
H^*(\spt)\!@>>>\!H^*(\sptn)
\end{CD}
\]
induced by restriction, for any $n\geq 1$. The homomorphisms 
$\pi^*$ are monic, because they induce isomorphisms over $\bZ[1/2]$. 
The upper restriction is epic by Theorem \ref{intcohbn}, and the lower 
by Corollary \ref{sptntf}; it is convenient to denote their kernels by
$K_n^*$ and $L_n^*$ respectively. 
\begin{thm}\label{intcohsptn}
For any $n\geq 1$, there are isomorphisms
\begin{align*}
H^*((\sptn)_{\ssp})&\;\cong\;
\bZ\opfm[\om h^p\!/2^{p-1}\!,\,g^qh^s\!/2^s\!,\,u_{i,j}]\,/\,J_n\\
&\;\cong\;H^*(\sptssp)/
(r_t,\,u_{i,t}:t>n)
\end{align*}
of graded rings, where $p,\om q\geq 1$, $s\geq 0$ and $0<i<jd\2/2$; 
the ideal $J_n$ is given by
\[
(2u_{i,j},\,u_{i,j}u_{k,l},\,u_{i,j}h^p\!/2^{p-1}\!,\,
u_{i,j}g^qh^s\!/2^s\!,\, r_t,\,u_{i,t}:t>n)\,,
\]
and the polynomials $r_t(g,h)$ by \eqref{ffdefr}.
\end{thm}
\begin{proof}
In odd dimensions, Proposition \ref{odcohspt} leads to an isomorphism
\[
H^{od}(\sptn)\;\cong\;\zmt\br{u_{i,j}:0<i<jd/2,\,j\leq n}
\]
by truncating $H^*(\KPI)$. It therefore remains to check that 
$L^{ev}_n$ and $R^{ev}_n$ (as introduced in Remarks \ref{rkinsp}) 
coincide in $H^{ev}(\spt)$.

By Theorem \ref{intcohbn} and Remark \ref{psprops}, $r_t$ lies in 
$K_n^{ev}$ for every $t>n$. So by Remarks \ref{rkinsp}, 
$r_t(g,h)$ lies in $L_n^{ev}$, and $R^{ev}_n\subseteq L^{ev}_n$. 
For the reverse inclusion, consider $r$ in $L_n^{kd}$ for any $k\geq 1$; 
then $\pi^*(r)$ lies in both $K_n^{kd}$ and $\Ima\pi^*$. Since 
$K_n^*=(r_{n+1},\om r_{n+2},\om y^{n+1})$, an expression of the form
\[
\pi^*(r)\;=\;v_1 r_{n+1}+v_2\om r_{n+2}+v_3\om y^{n+1}
\]
must hold for some polynomials $v_j=v_j(m,y)$ in $Z^*$, where 
$j=1$, $2$, or $3$. But $v_1 r_{n+1}+v_2 r_{n+2}$ lies in 
$\pi^*(R^{kd}_n)$ by Remarks \ref{rkinsp}, so $v_3y^{n+1}$ lies in 
$\Ima\pi^*$ as well. Thus $v_3y^{n+1}\equiv 2\lambda y^{kd/2}$
or $0$ mod $(m)$ for even or odd $k$ respectively, and therefore lies 
in $\pi^*(R^{kd}_n)$, by Remarks \ref{rkinsp}(2) and (3). So $\pi^*(r)$ 
lies in $\pi^*(R^{kd}_n)$, and $r$ lies in $R_n^{kd}$; hence 
$L^{ev}_n\subseteq R^{ev}_n$, as sought.
\end{proof}

Combining \eqref{ffdefr} and Remarks \ref{rkinsp} gives
$r_{2k}(g,h)\equiv(-1)^kh^k/2^{k-1}$ mod $(g)$ and
$r_{2k+1}(g,h)\equiv (-1)^k(2k\?+\?1)gh^k/2^k$ mod $(g^2)$, for $k\geq
0$.  Corollary \ref{bntrunc} identifies $H^*(\sptn)$ as
$\bZ\br{h^n/2^{n-1}}$ in dimension $2nd$, and $0$ above.

%
%
%
%
%
%
%
%
%

\setcounter{equation}{0}
\renewcommand{\theequation}{A.\arabic{equation}}

%
%
%
%
%
%
%
%
%

\section*{Appendix A}\label{ap}

In these tables, the notation is that of \eqref{isobst}, \eqref{cohthsp}, 
\eqref{defuij} and \eqref{defgh}. The additive generators exhibit the 
multiplicative structure in each dimension, and monomials in $g$ and
$h$ have infinite order; all others have order $2$. 

\begin{exa}\label{infmod2}
By Theorem \ref{cohsptmod2},
$H^*(\spt(\CPI);\zmt)$ is given by 
\begin{equation*}\label{cplxinfzmt}
\begin{tabular}{|c|c|c|c|c|c|c|c|c|}
\hline
\rule{0em}{.9em}0&2&4&6&7&8&9\\
\hline
\rule{0em}{1.1em}$1$&$t$&$a^2\om t,\om m\om t$&
$a^4\om t,\om y\om t,\om m^2\om t$&$a^5\om t$&
$a^6\om t,\om a^2yt,\om my\om t,\om m^3t$&$\om a^3y\om t$\\
\hline 
\end{tabular}
\end{equation*}
\rule[-.5em]{0em}{1.5em} 
in dimensions $\leq 9$, and $H^*(\spt(\HPI);\zmt)$ is given by 
\begin{equation*}\label{quatinfzmt}
\begin{tabular}{|c|c|c|c|c|c|c|c|c|}
\hline
\rule{0em}{.9em}0&4&6&7&8&10&11&12&13\\
\hline
\rule{0em}{1.1em}$1$&$t$&$a^2\om t$&$a^3\om t$&$a^4\om t,\om m\om t$&
$a^6\om t$&$a^7\om t$&$a^8\om t,\om y\om t,\om m^2\om t$&$a^9\om t$\\
\hline 
\end{tabular}
\end{equation*}
\[
\begin{tabular}{|c|c|c|c|c|}
\hline
\rule{0em}{.9em}14&15&16&17&18\\
\hline
\rule{0em}{1.1em}$a^{10}\om t,\om a^2y\om t$&
$a^{11}\om t,\om a^3y\om t$&
$a^{12}\om t,\om a^4y\om t,\om my\om t,\om m^3\om t$&$a^5y\om t$&
$a^{14}\om t,\om a^6y\om t$\\
\hline 
\end{tabular}\,.
\]
\rule{0em}{1.1em}\noindent 
in dimensions $\leq 18$
\end{exa}

\begin{exa}\label{infint}
By Theorem \ref{intcohspt}, $H^*(\spt(\CPI))$ is given by 
 \rule[-.5em]{0em}{.5em}
\[
\begin{tabular}{|c|c|c|c|c|c|c|c|c|}
\hline
\rule{0em}{.9em}0&2&4&6&7&8&9&10&11\\
\hline
\rule{0em}{1.1em}$1$&$g$&$g^2,\om h$&
$g^3,\om gh\?/\2 2$&
$u_{1,2}$&$g^4,\om g^2h\?/\2 2,\om h^2\?/\2 2$&
$u_{1,3}$&$g^5,\om g^3h\?/\2 2,\om gh^2\?/\2 4$
&$u_{2,3},\om u_{1,4}$\\
\hline 
\end{tabular}
\]
\rule[-.5em]{0em}{1.5em}in dimensions $\leq 11$, and
$H^*(\spt(\HPI))$ is given by
\[
\begin{tabular}{|c|c|c|c|c|c|c|c|c|c|}
\hline
\rule{0em}{.9em}0&4&7&8&11&12&13&15&16&17\\
\hline
\rule{0em}{1.1em}$1$&$g$&$u_{1,1}$&$g^2,\om h$&$u_{1,2}$&
$g^3,\om gh\?/\2 2$&$u_{2,2}$&$u_{3,2},\om u_{1,3}$&
$g^4,\om g^2h\?/\2 2,\om h^2\?/\2 2$&$u_{2,3}$\\
\hline 
\end{tabular}
\]
\[
\begin{tabular}{|c|c|c|c|c|}
\hline
\rule{0em}{.9em}19&20&21&23&24\\
\hline
\rule{0em}{1.1em}$u_{3,3},\om u_{1,4}$&
$g^5,\om g^3h\?/\2 2,\om gh^2\?/\2 4$&$u_{4,3},\om u_{2,4}$&
$u_{5,3},\om u_{3,4},\om u_{1,5}$&
$g^6,\om g^4h\?/\2 2,\om g^2h^2\?/\2 4,\om h^3\?/\2 4$\\
\hline 
\end{tabular}
\]
\rule{0em}{1.3em}
 in dimensions $\leq 24$. 
\end{exa}

\begin{exa}\label{finmod2}
By Theorem \ref{cohsptnmod2}, $H^*(\spt(\CPT);\zmt)$ is given by 
\begin{equation*}\label{cplxlowdimzmt}
\begin{tabular}{|c|c|c|c|c|c|c|c|}
\hline
\rule{0em}{.9em}0&2&4&6&7&8\\
\hline
\rule{0em}{1.1em}$1$&$t$&$\!a^2\om t,\om m\om t\!$&
$\!a^4\om t,\om m^2\om t\!$&$a^5\om t$&
$\!a^6\om t=m^3\om t\!$\\
\hline 
\end{tabular}
\end{equation*}
\rule[-.5em]{0em}{1.5em} 
and $H^*(\spt(\HPT);\zmt)$ is given by
\begin{equation*}\label{quatlowdimzmt}
\begin{tabular}{|c|c|c|c|c|c|c|c|c|c|c|c|}
\hline
\rule{0em}{.9em}0&4&6&7&8&10&11&12&13&14&15&16\\
\hline
\rule{0em}{1.1em}$1$&$t$&$a^2\om t$&$a^3\om t$&$\!a^4\om t,\om m\om t\!$&
$a^6\om t$&$a^7\om t$&$\!a^8\om t,\om m^2\om t\!$&$a^9\om t$&$a^{10}\om t$&
$a^{11}\om t$&$\!a^{12}\om t=m^3\om t\!$\\
\hline 
\end{tabular}\,,
\end{equation*}
\rule{0em}{1.1em}\noindent 
where $t^2=mt$ in both cases.
\end{exa} 

\begin{exa}\label{finint}
By Theorem \ref{intcohsptn}, $H^*(\spt(\CPH))$ \mbox{is given by} 
\[
\begin{tabular}{|c|c|c|c|c|c|c|c|c|c|}
\hline
\rule{0em}{.9em}0&2&4&6&7&8&9&10&11&12\\
\hline
\rule{0em}{1.1em}$1$&$g$&$g^2,\om h$&
$g^3,\om gh\?/\2 2$&
$u_{1,2}$&$g^2h\?/\2 2,\om h^2\?/\2 2$&
$u_{1,3}$&$gh^2\?/\2 4$&$u_{2,3}$&$h^3\?/\2 4$\\
\hline 
\end{tabular}
\]
\rule[-.5em]{0em}{1.5em} 
and $H^*(SP^2(\HPH))$ is given by
\[
\begin{tabular}{|c|c|c|c|c|c|c|c|c|c|}
\hline
\rule{0em}{.9em}0&4&7&8&11&12&13&15&16&17\\
\hline
\rule{0em}{1.1em}$1$&$g$&$u_{1,1}$&$g^2,\om h$&$u_{1,2}$&
$g^3,\om gh\?/\2 2$&$u_{2,2}$&$u_{3,2},\om u_{1,3}$&
$g^2h\?/\2 2,\om h^2\?/\2 2$&$u_{2,3}$\\
\hline 
\end{tabular}
\]
\[
\begin{tabular}{|c|c|c|c|c|}
\hline
\rule{0em}{.9em}19&20&21&23&24\\
\hline
\rule{0em}{1.1em}$u_{3,3}$&$gh^2\?/\2 4$&$u_{4,3}$&
$u_{5,3}$&$h^3\?/\2 4$\\
\hline 
\end{tabular}\,,
\]
\rule{0em}{1.3em}
with relations $g^4=4g^2h/2-h^2\!/2$ and $g^3h/2=3gh^2\!/4$ in both 
cases.
\end{exa}

\end{document}